\documentclass[11pt,double,reqno]{amsart}

\usepackage{tikz}
\usepackage{times}
\usepackage{graphicx}
\usepackage{amssymb}
\usepackage{epsfig}
\usepackage{amsmath}
\usepackage{amsthm}
\usepackage{color}

\usepackage{amsmath, amssymb}
\usepackage{enumerate}

\usepackage[bitstream-charter]{mathdesign}
\usepackage[T1]{fontenc}

\newtheorem{theo}{Theorem}[section]
\newtheorem{defin}[theo]{Definition}
\newtheorem{prop}[theo]{Proposition}

\newtheorem{lemm}[theo]{Lemma}
\newtheorem{rem}[theo]{Remark}

\numberwithin{equation}{section}

\newcommand{\al}{\alpha}
\newcommand{\be}{\beta}
\newcommand{\ga}{\gamma}
\newcommand{\Ga}{\Gamma}
\newcommand{\la}{\lambda}

\newcommand{\om}{\omega}
\newcommand{\Om}{\Omega}

\newcommand{\si}{\sigma}

\newcommand{\ep}{\epsilon }
\newcommand{\te}{\theta}
\newcommand{\De}{\Delta}
\newcommand{\de}{\delta}

\newcommand{\pa}{\partial}

\newcommand{\R}{{\mathbb R}^n}

\newcommand{\ri}{\rightarrow}
\newcommand{\Rn}{{\mathbb R}^{n-1}}

\newcommand{\na}{\nabla}
\newcommand{\RN}{{\mathbb R}^{n+1}}
\begin{document}
\baselineskip=18pt

\title[]{Global existence of    non-Newtonian incompressible fluid  in half space with nonhomogeneous initial-boundary data}

\
\author{Tongkeun Chang}
\address{Department of Mathematics, Yonsei University \\
Seoul, 136-701, South Korea}
\email{chang7357@yonsei.ac.kr}

\author{BumJa Jin}
\address{Department of Mathematics, Mokpo National University, Muan-gun 534-729,  South Korea }
\email{bumjajin@mokpo.ac.kr}

\begin{abstract}
In this  study, we investigate  the global existence of weak solutions  of non-Newtonian incompressible fluids governed by  \eqref{maineq2}. When $u_0 \in \dot   B^{\alpha-\frac{2}{p}}_{p,q}({\mathbb R}^{n}_+) \, \cap  \,\dot  B^{ 1 -\frac4{n+2}}_{\frac{n+2}2,\frac{n+2}2}({\mathbb R}^{n}_+) \,\cap \, \dot B^{ 1 +\frac{n}p}_{p,1} (\R_+)$ is given, we will find the weak solutions for the equation \eqref{maineq2} in the function space  $C_b ([ 0, \infty; \dot B^{\alpha -\frac2p}_{p,q} ({\mathbb R}^n_+)) \cap C_b (0, \infty; \dot B^{1 -\frac4{n+2}}_{\frac{n+2}2} (\R_+)) \cap L^\infty(0, \infty; \dot W^1_\infty(\R_+))$, $ n+2 < p < \infty, \,\,  1 \leq q \leq \infty, \,\, 1 + \frac{n+2}p < \alpha < 2$.  We show the existence of weak solutions in the  anisotropic  Besov spaces $\dot B^{\alpha, \frac{\alpha}2}_{p,q} (\R_+ \times (0, \infty))$ (see Theorem \ref{thm-navier}) and we   show the embedding   $\dot B^{\alpha, \frac{\alpha}2}_{p,q} (\R_+ \times (0, \infty) \subset C_b ([ 0, \infty; \dot B^{\alpha -\frac2p}_{p,q} ({\mathbb R}^n_+))$ (see Lemma \ref{lemma1208}). 
 
 For the global existence of solutions, we assume that the extra stress tensor $S$ is represented by $S({\mathbb A}) = {\mathbb F} ( {\mathbb A}) {\mathbb A}$, where   ${\mathbb F}(0) $ is  a uniformly elliptic matrix and  $ {\mathbb F} \in C^2(B(0,1))$, where $B(0,1)$ is open ball in ${\mathbb R}^{n\times n}$ whose center is origin and radius. is $1$. 
Note that  $S_1$, $S_2$ and $S_3$  introduced in \eqref{0207-1} satisfy our assumptions.\\

\noindent
 2000  {\em Mathematics Subject Classification:}  primary 35K61, secondary 76D07. \\

\noindent {\it Keywords and phrases: Non-Newtonian equations,  initial-boundary value,  half space. }

\end{abstract}

\maketitle

\section{\bf Introduction}
\setcounter{equation}{0}

In this   study, we investigate  the global existence of solutions for a   non-Newtonian incompressible fluid governed by the following system:  
\begin{align}\label{maineq2}
\begin{array}{l}\vspace{2mm}
u_t - {\rm div  }\,\big( S(Du) \big) + \na p =-{\rm div}(u\otimes u), \qquad {\rm div} \, u =0 \mbox{ in }
 \R_+ \times (0, \infty),\\
\hspace{30mm}u|_{t=0}= u_0, \quad u |_{x_n =0} = 0,
\end{array}
\end{align}
where $u_0$ is the   given initial  velocity and
 $u$ and $p$ are the unknown velocity and pressure, respectively.         Here, $ (D u)_{ij} = \frac12 ( \frac{\pa u_i}{\pa x_j} + \frac{\pa u_j}{\pa x_i})$, $ i, \, j = 1, 2, \cdots, n$ and $ S: \R \times \R \ri \R \times \R $ is the extra stress tensor. When $ S(Du) = 2 Du$, the equations \eqref{maineq2} become usual Navier-Stokes equations.

The standard models of $S(Du)$ are as follows:
\begin{align}\label{0207-1}
\notag S_1(Du)& = (\mu_0  +  \mu_1|Du|)^{d -2}Du, \\
\notag S_2(Du)& =  (\mu_0  + \mu_1 |Du|^2)^{\frac{d-2}2}Du,\\
S_3(Du) &= (\mu_0 + \mu_1|Du|^{d-2} ) Du,
\end{align}
where $\mu_0, \, \mu_1 >0$ are positive real numbers. Here,  $|{\mathbb A}| = \big( \sum_{1\leq i,j \leq n} |a_{ij}|^2  \big)^\frac12 $ for matrix ${\mathbb A} = (a_{ij})_{1 \leq ij \leq n}$.   We call the fluid shear-thickening   if $d > 2$ and shear-thinning   if $1 < d < 2$.

We review  the prior results regarding for the existence which are relevant to  our result. Lady$\check{\rm z}$henskaja (\cite{La0,La}) was the first to study the existence of weak solutions of \eqref{maineq2} in bounded domains with  no-slip boundary condition $( u|_{x_n} =0)$,  and  Nečas et al. also  studied this problem intensively. After them, the non-Newtonian incompressible fluids were studied by many mathematicians ( see \cite{BP, BK, DR,KMS, Pok} pertaining to the whole space and see \cite{Am, BD, BW,FMS, KKK, MNR, V, V2, VKR, Wol} for bounded domains  with no-slip boundary condition).

In the engineering
literatures,  the non-Newtonian fluids  under the no-slip condition are no longer well accepted,
as shown by  the extensive and thorough discussion in \cite{LBS}.

%Until now, the results are mainly induced from non-slip boundary condition.
%
% we can not find the Dirichlet boundary problem for non-Newtonian Navier-Stokes equations.

In smooth bounded domains,  the pure slip boundary condition case
%\begin{align*}
%  u \cdot {\bf n} =0, \quad   S(Du)Du {\bf n} - (S(Du) Du   {\bf n} ) {\bf n} =0
%\end{align*}
was  studied by   Bulíček,   Málek  and   Rajagopal \cite{BMR} (also see \cite{MS}) and  the pure Neumann boundary condition case
%\begin{align*}
%S(Du) Du {\bf n} =0
%\end{align*}
was  studied by  Bothe and  Prüss \cite{BP} (also see \cite{S}).%, where ${\bf n}$ is outer normal vector on boundary of %domain.

In this study, we investigage the nonhomogeneous boundary problem for non-Newtonian imcompressible fluid ( $u|_{x_n =0}  \neq 0$). Over the past decade many mathematicians have studied  the Newtonian incompressible fluids   based on nonhomogeneous boundary data 
(See \cite{fernandes,amann1, amann2, CJ1, chang-jin, CJ2, farwig4,farwig6, farwig2, farwig3,grubb1, grubb2, grubb3, grubb,lewis,R, voss} and the references therein).  

Recently, Bae and Kang \cite{BK} showed the existence of a low  regular solutions of \eqref{maineq2} in $\R$ with assumptions $ S(Du) = {\mathbb F} (Du) Du$,  ${\mathbb F}({\mathbb A}) = F (|{\mathbb A}|^2)$, $ F(0) ={\mathbb I}$  and $F \in C^2([0, \infty))$, where ${\mathbb I}$ is identity matrix in $\R$.

Motivated by Bae and Kang's results,  in this study we  show the existence of  a weak solutions with minimal  regularity and decay conditions of $S(Du)$.  

We introduce our assumptions for $S(Du)$. We assume   that $S (Du)$ is represented by $S(Du) = {\mathbb F} ( Du) Du$, where ${\mathbb F} : \R \times \R \ri  {\mathbb R}^{n \times n} $ is satisfied   that \\

{(\bf A)} ${\mathbb F} \in C^2 ( B(0,1))$, where $B(0,1)$ is open unit ball in ${\mathbb R}^{n \times n}$ and  ${\mathbb F}( 0) $ is a uniformly elliptic matrix. \\

%{\bf A2)} For a  given  $\ep > 0$,  there is  $\de > 0$,  such that   
%\begin{align}\label{assumption}
% |\big( {\mathbb F}({\mathbb A}) -{\mathbb F}(0) \big) {\mathbb A} - \big( {\mathbb F}({\mathbb B})- {\mathbb F}(0) \big){\mathbb B}| \leq \ep  |{\mathbb A} -{\mathbb B}|   \quad \mbox{for all $|{\mathbb A}|, \,\, | {\mathbb B}| < \de$}.
%  \end{align}
% If ${\mathbb F} \in C^1 (B(0,1) \setminus \{ 0\} ) \, \cap \,  C (B(0,1) ) $, then  ${\mathbb F}$ satisfies assumption \eqref{assumption}, where $B(0,1)$ is open unit ball in ${\mathbb R}^{n \times n}$.  For example, ${\mathbb F} ({\mathbb A}) =\big(\frac1{ \ln \, |{\mathbb A}|} +1) {\mathbb I}$ if ${\mathbb A} \neq 0$ and  ${\mathbb F} (0) ={\mathbb I}$   satisfies the above conditions. 
Note that    $S_1$, $S_2$ and $S_3$  introduced in \eqref{0207-1} satisfy our assumptions with 
$   F_1({\mathbb A}) =    (\mu_0  +  \mu_1|{\mathbb A}|)^{d -2}  {\mathbb I},$ $   F_2({\mathbb A}) =  (\mu_0  + \mu_1 |{\mathbb A}|^2)^{\frac{d-2}2}  {\mathbb I}$ and $
F_3({\mathbb A}) = (\mu_0 + \mu_1|{\mathbb A}|^{d-2} )  {\mathbb I}$.

Our main results are follows:
\begin{theo}
\label{thm-navier-2}
Let $S (Du) = {\mathbb F}(Du) Du$ and ${\mathbb F} : \R \times \R \ri  {\mathbb R}^{n \times n} $ satisfy the assumptions {(\bf A)}.   Let $n+2 < p< \infty$ and $ 1 \leq q \leq \infty$ such that  $1 + \frac{n+2}p < \al  < 2$.  Let
\begin{align*}
u_0 &  \in  \dot   B^{\al-\frac{2}{p}}_{p,q}({\mathbb
R}^{n}_+) \, \cap  \,\dot  B^{ 1 -\frac4{n+2}}_{\frac{n+2}2,\frac{n+2}2}({\mathbb R}^{n}_+) \,\cap \, \dot B^{ 1 +\frac{n}p}_{p,1} (\R_+)
\end{align*}
such that $ {\rm div} \, u_0 =0$ and $ u_0|_{\Rn \times (0, \infty)} =0$.  There is $\de > 0 $  such that if
\begin{align*}
&  \| u_0\|_{\dot  B^{1-\frac{4}{n+2}}_{\frac{n+2}2,\frac{n+2}2}({\mathbb R}^{n}_+)  }  + \| u_0 \|_{\dot B^{ 1 +\frac{n}p}_{p,1} (\R_+) }  + \| u_0\|_{\dot   B^{\al-\frac{2}{p}}_{p,q}({\mathbb R}^{n}_+)}    < \de,
\end{align*}
then,  \eqref{maineq2} has a solution $u$ satisfying 
\begin{align*}
 u\in    C_b ([0, \infty); \dot B^{\al -\frac2p}_{p,q} (\R_+))\, \cap \,   C_b([0, \infty); \dot W^{ 1 -\frac4{n+2}  }_{\frac{n+2}2 } ({\mathbb R}^{n}_+ ) ) \cap L^\infty(0,\infty, \dot W^1_\infty(\R_+) ).
\end{align*}
\end{theo}

To prove  Theorem \ref{thm-navier-2},  we show the existence of the  solutions of  the equation \eqref{maineq2} in anisotropic Besov spaces  $  \dot B^{\al,\frac{\al}2}_{p,q} (\R_+ \times (0, \infty))$ (see Lemma \ref{lemma1208}).   In showing the existence of the  solutions,   the anisotropic Besov spaces   have several merits compared to the function spaces  $C_b ([0, \infty); \dot B^{\al -\frac2p}_{p,q} (\R_+))$. The  parabolic type singular intergral operators are bounded in  $ \dot B^{\al,\frac{\al}2}_{p,q} (\R_+ \times (0, \infty))$  but  not in  $C_b ([0, \infty); \dot B^{\al -\frac2p}_{p,q} (\R_+))$.  This means that  we can represent the solutions with  the integral forms and using it  we can obtain $ \dot B^{\al,\frac{\al}2}_{p,q} (\R_+ \times (0, \infty))$-norm estimates of solutions (see Section \ref{Proof1}).

We consider equations \eqref{maineq2}  with non-zero boundary condition, $ u|_{x_n =0}  \not= 0$.   Theorem \ref{thm-navier-2} is special case of the following results.
\begin{theo}
\label{thm-navier}
Let $S (Du) = {\mathbb F}(Du) Du$ and ${\mathbb F}: \R \times \R \ri  {\mathbb R}^{n \times n} $ satisfy the assumptions {\bf (A)}.  Let $n+2 < p< \infty$ and $ 1 \leq q \leq \infty$ such that  $1 + \frac{n+2}p < \al  < 2$.  Let
\begin{align*}
u_0 &  \in  \dot   B^{\al-\frac{2}{p}}_{p,q}({\mathbb
R}^{n}_+) \, \cap  \,\dot  B^{ 1 -\frac4{n+2}}_{\frac{n+2}2,\frac{n+2}2}({\mathbb R}^{n}_+) \,\cap \, \dot B^{ 1 +\frac{n}p}_{p,1} (\R_+),\\
g & \in   \dot B^{ \al-\frac{1}{p},\frac{\al} 2-\frac{1}{2p}}_{p,q}({\mathbb R}^{n-1} \times (0, \infty)) \,\cap \,\dot B^{ 1-\frac{2}{n+2},\frac{1}2-\frac{1}{n+2}}_{\frac{n+2}2 , \frac{n+2}2}({\mathbb R}^{n-1} \times (0, \infty))\, \cap  \,\dot  B^{ 1 + \frac{n+1}p,\frac12 +\frac{n+1}{2p}}_{p,1}({\mathbb R}^{n-1} \times (0, \infty)), \\
g_n  & \in \dot A^{\al -\frac1p, \frac{\al}2 -\frac1{2p} }_{p,q} ( \Rn \times (0, \infty))\, 
\cap \, \dot A^{1-\frac2{n+2}, \frac12 -\frac1{n+2} }_{\frac{n+2}2,\frac{n+2}2} ( \Rn \times (0, \infty) ) \, 
\cap  \,\dot  A^{ 1 +\frac{n+1}p, \frac12 +\frac{n+1}{2p}}_{p,q} ( \Rn \times (0, \infty)  ) 
\end{align*}
such that $ u_0$ and $g$ satisfy  the following conditions;
 \begin{align}\label{compatibility}
 {\rm div }\, u_0 =0, \quad  \mbox{and} \quad  g|_{t =0} = u_0|_{x_n=0}.
 \end{align}
If
\begin{align*}
&  \| u_0\|_{\dot  B^{1-\frac{4}{n+2}}_{\frac{n+2}2,\frac{n+2}2}({\mathbb R}^{n}_+)  }  + \| u_0 \|_{\dot B^{ 1 +\frac{n}p}_{p,1} } 
  +  \| u_0\|_{\dot   B^{\al-\frac{2}{p}}_{p,q}({\mathbb R}^{n}_+)}\\
  &   +  \| g \|_{\dot B^{ 1-\frac{2}{n+2},\frac{1}2-\frac{1}{n+2}}_{\frac{n+2}2,\frac{n+2}2 }({\mathbb R}^{n-1} \times (0, \infty))  } + \| g\|_{ \dot  B^{ 1 + \frac{n +1}p, \frac12 +\frac{n+1}{2p}}_{p,1}({\mathbb R}^{n-1} \times (0, \infty))} + \| g \|_{ \dot B^{ \al-\frac{1}{p},\frac{\al} 2-\frac{1}{2p}}_{p,q}({\mathbb R}^{n-1} \times (0, \infty))  } \\
  & \qquad + \| g_n \|_{A^{1-\frac2{n+2}, \frac12 -\frac1{n+2} }_{\frac{n+2}2,\frac{n+2}2} ( \R \times (0, \infty) ) } + \| g_n \|_{ \dot  A^{ 1 +\frac{n+1}p, \frac12 +\frac{n+1}{2p}}_{p,q} ( \R \times (0, \infty)  ) }  +  \| g_n \|_{\dot A^{\al -\frac1p, \frac{\al}2 -\frac1{2p} }_{p,q} ( \Rn \times (0, \infty))}\\
  & < \de
\end{align*}
for small $\de > 0$,  then the equations   \eqref{maineq2} with the boundary condition $u|_{x_n =0} = g$ has a solution $u$ satisfying 
\begin{align}
 u\in    \dot B^{\al,\frac{\al}2}_{p,q} (\R_+ \times (0, \infty)) \, \cap \,   \dot W^{ 1, \frac12   }_{\frac{n+2}2 } ({\mathbb R}^{n}_+ \times (0, \infty ) ) \cap L^\infty(0, \infty; \dot W^1_\infty(\R_+)).
 \end{align}
 Moreover,  there is $\de_0> 0$ such that if $u_1$ is another solution of \eqref{maineq2} satisfying $ u_1 \in \dot W^{ 1, \frac12   }_{\frac{n+2}2 } ({\mathbb R}^{n}_+ \times (0, \infty ) ) \cap L^\infty(0, \infty; \dot W^1_\infty(\R_+))$ with  $ \| u_1\|_{\dot W^{ 1, \frac12   }_{\frac{n+2}2 } ({\mathbb R}^{n}_+ \times (0, \infty ) ) } + \| Du_1\|_{L^\infty (\R_+ \times (0, \infty))} < \de_0$, then $ u = u_1$.

\end{theo}
The function spaces are introduced in Section 2.

We note that the homogeneous Besov spaces $ \dot B^{\al, \frac{\al}2}_{p,q} (\R_+ \times (0, \infty))$ are not complete if  $ \al > \frac{n+2}p$.  
Because we use the iteration method to find solutions, it looks like  the function space   $\dot B^{\al,\frac{\al}2}_{p,q} (\R_+ \times (0, \infty)), \,\, \al > 1 +  \frac{n+2}p$ are not reasonable for the  solution spaces. To overcome it, we add the complete  function space  $  \dot W^{ 1, \frac12   }_{\frac{n+2}2 } ({\mathbb R}^{n}_+ \times (0, \infty ) )$ (see Theorem \ref{theo0715-1}). Our approximate solutions $u^m$ are uniformly  bounded and Cauchy sequence  in  $\dot B^{\al,\frac{\al}2}_{p,q} (\R_+ \times (0, \infty))\cap  \dot W^{ 1, \frac12   }_{\frac{n+2}2 } ({\mathbb R}^{n}_+ \times (0, \infty ) )$. Then, by completeness of $ \dot B^{\al,\frac{\al}2}_{p,q} (\R_+ \times (0, \infty))\cap \dot W^{ 1, \frac12   }_{\frac{n+2}2 } ({\mathbb R}^{n}_+ \times (0, \infty ) )$, $u^m$ converges to $u$ in $ \dot B^{\al,\frac{\al}2}_{p,q} (\R_+ \times (0, \infty))\cap\dot W^{ 1, \frac12   }_{\frac{n+2}2 } ({\mathbb R}^{n}_+ \times (0, \infty ) )$.

We note that the restriction $\al < 2$ (regularity of solutions) is based on the assumptions of $S(Du)$. With an improvement in the regularity of ${\mathbb F}$,  the regularity of the solutions in  Theorem  \ref{thm-navier} will be also improved (see \cite{KKK}).

For simplicity, we assume that ${\mathbb F} (0) = {\mathbb I}$ is the  identity matrix.  Let 
\begin{align}\label{0501-2}
\si ({\mathbb A}) : = {\mathbb F} ({\mathbb A}) - {\mathbb F}(0).
\end{align}
Accordingly, the first equations in \eqref{maineq2} are denoted by
\begin{align*}
u_t  -\De u  +\na p= {\rm div  }\,\big( \si(Du) D u - u \otimes u\big).
\end{align*}

Thus, for the proof of Theorem \ref{thm-navier}, the following initial-boundary value problem of the Stokes equations in $\R_+\times (0,\infty)$ is necessary:
\begin{align}\label{maineq-stokes}
\begin{array}{l}\vspace{2mm}
w_t - \De w + \na p =f, \qquad {\rm div} \, w =0 \mbox{ in }
 \R_+\times (0,\infty),\\
\hspace{30mm}w|_{t=0}= u_0, \qquad  w|_{x_n =0} = g,
\end{array}
\end{align}
where $f=\mbox{div}{\mathcal F}$($= {\rm div} \, \big( \si(Du) Du - u \otimes u  \big)$.

The following theorem states our results on the unique solvability of the Stokes equations \eqref{maineq-stokes}.
\begin{theo} 
\label{thm-stokes}
Let $1 < p< \infty$, $ 1 \leq q \leq \infty$  and   $1  <   \al  <  2$.
\begin{itemize}
\item[(1)]
Let 
\begin{align*}
u_0  \in  \dot   B^{\al-\frac{2}{p}}_{p,q}({\mathbb
R}^{n}_+), \quad  {\mathcal F}   \in L^p (0, \infty; \dot B^{\al-1 }_{p,q}(\R_+ )),\\
g  \in   \dot B^{ \al-\frac{1}{p},\frac{\al} 2-\frac{1}{2p}}_{p,q}({\mathbb R}^{n-1} \times (0, \infty)), \quad 
g_n   \in \dot A^{\al -\frac1p, \frac{\al}2 -\frac1{2p} }_{p,q} ( \Rn \times (0, \infty))
\end{align*}
such that $ u_0$ and $g$ satisfy \eqref{compatibility}.
Accordingly,  equation \eqref{maineq-stokes} has a solution
 $w\in   \dot B^{\al, \frac{\al}2 }_{p,q} (\R_+ \times (0, \infty))$   with the following inequalities
\begin{align} \label{240718-1-1}
\begin{split}
  \| w\|_{\dot B^{\al, \frac{\al}2}_{p,q}(\R_+ \times (0, \infty))}
& \leq c \big(  \| u_0\|_{ \dot B_{p,q}^{\al -\frac2p} (\R_+)}  + \| g\|_{  \dot B^{ \al-\frac{1}{p},\frac{\al} 2-\frac{1}{2p}}_{p,q}({\mathbb R}^{n-1} \times (0, \infty)) }\\
& \quad  + \| g_n \|_{\dot A^{\al -\frac1p, \frac{\al}2 -\frac1{2p} }_{p,q} ( \Rn \times (0, \infty))}   + \| {\mathcal F}\|_{\dot B^{\al-1,\frac{\al }2 -\frac12 }_{p,q}(\R_+ \times (0 ,\infty ))}\big).
\end{split}
\end{align}

\item[(2)]
Let 
\begin{align*}
u_0  \in  \dot   B^{1-\frac{2}{p}}_{p,p}({\mathbb
R}^{n}_+), \quad  {\mathcal F}   \in L^p ( \R_+ \times (0, \infty )),\\
g  \in   \dot B^{ 1-\frac{1}{p},\frac{1} 2-\frac{1}{2p}}_{p,p}({\mathbb R}^{n-1} \times (0, \infty)), \quad 
g_n   \in \dot A^{1 -\frac1p, \frac{1}2 -\frac1{2p} }_{p,p} ( \Rn \times (0, \infty))
\end{align*}
such that $ u_0$ and $g$ satisfy \eqref{compatibility}.
Then,   equation \eqref{maineq-stokes} has a  solution
 $w\in   \dot W^{1, \frac{1}2 }_{p} (\R_+ \times (0, \infty))$   with the following inequalities
\begin{align} \label{240718-2-2}
\begin{split}
  \| w\|_{\dot W^{1, \frac{1}2}_{p}(\R_+ \times (0, \infty))}
& \leq c \big(  \| u_0\|_{ \dot B_{p,p}^{1-\frac2p} (\R_+)}  + \| g\|_{  \dot B^{ 1-\frac{1}{p},\frac{1} 2-\frac{1}{2p}}_{p,p}({\mathbb R}^{n-1} \times (0, \infty)) }\\
& \quad  + \| g_n \|_{\dot A^{1 -\frac1p, \frac{1}2 -\frac1{2p} }_{p,p} ( \Rn \times (0, \infty))}   + \| {\mathcal F}\|_{L^p ( \R_+ \times (0, \infty ))}\big).
\end{split}
\end{align}
In particular, if $ p < n+2$, then the solution $w $ of \eqref{maineq-stokes} is a unique in $ \dot W^{1, \frac{1}2 }_{p} (\R_+ \times (0, \infty))$.

\end{itemize}
\end{theo}

We organize this paper as follows.
In Section \ref{notation}, we introduce  the function spaces, and the  definitions of the weak solutions of Stokes equations and non-Newtonian Navier-Stokes equations.  In Section \ref{preliminary},   the various  estimates of operators related to  Newtonian   and Gaussian kernels are provided.  
In Section \ref{Proof1}, we  complete the proof of Theorem \ref{thm-stokes}.  In Section \ref{nonlinear}, we provide the proof of Theorem \ref{thm-navier}  by applying the estimates in  Theorem \ref{thm-stokes} to the  approximate solutions.

\section{Notations and Definitions}
\label{notation}
\setcounter{equation}{0}
The points of spaces $\Rn$ and $\R$ are denoted by $x'$ and $x=(x',x_n)$, respectively.
In addition, multiple derivatives are denoted by $ D^{\be}_x D^{m}_t = \frac{\pa^{|\be|}}{\pa x^{\be}} \frac{\pa^{m} }{\pa t}$ for multi-index
$ \be$ and nonnegative integer $ m$. Throughout this paper, we denote various generic constants by using $c$.
Let $\R_+=\{x=(x',x_n): x_n>0\}$ and  $\overline{\R_+}=\{x=(x',x_n): x_n\geq 0\}$.

For the Banach space $X$, we denote the  usual Bochner space by $L^p(0, \infty;X), 1\leq p\leq \infty$.
For $0< \theta<1$ and $1\leq p\leq \infty$,  we denote   the complex interpolation and  the real interpolation  space of the normed space $X$ and $Y$ as $[X, Y]_\te$ and   $(X,Y)_{\theta,p}$, respectively.
For $1\leq p\leq \infty$,  $p'=\frac{p}{p-1}$.

 Let   $\Omega$ be either    $\R$, $\Rn$, or $\R_+$.
 Let $1\leq p\leq \infty$ and $k$ be a nonnegative integer.
 The norms of usual  Lebesgue space $L^p(\Omega)$, the usual   homogeneous Sobolev space   $  \dot W^k_p(\Omega)$  and  the  usual homogeneous  Besov space   $ \dot B^s_{p,q} (\Om), \,  1\leq p\leq \infty$  are written as $\|\cdot\|_{L^p(\Omega)}$,  $\|\cdot\|_{ \dot W^k_p(\Omega)}$ and $\|\cdot\|_{ \dot B^k_{p,q}(\Omega)}$, respectively.

%{\color{red}{
% It is known that
%  $ B^{s}_p(\Omega)=(L^p(\Omega),  W^{k}_p(\Omega))_{\frac{s}{k}, p}$ for $0<s<k$ and
% $ B^{s}_p(\Omega)=( B^{s_1}_p(\Omega),  B^{s_2}_p(\Omega))_{\theta, p}$
% for $s=(1-\theta)s_1+\theta s_2$, $0<\theta<1$  and $1<p<\infty$.
%}}

%We also  denote by   $\dot {B}^s_{p}(\Rn)$ and ${B}^s_{p}(\Rn)$  %the restriction of  $\dot {B}^s_{p}({\mathbb R}^n)$ and %${B}^s_{p}({\mathbb R}^n)$ over $\Rn$, respectively.

%Let $\Om $ be $\R$ or $\R_+$.
%Denote by $ \dot B^s_{p\sigma}(\Om) = \{ f \in  \dot B^s_{p} (\Om) \, | \, {\rm div} \, f =0 \}$ and $  B^s_{p\sigma}(\Om) = \{ f \in  B^s_{p} (\Om) \, | \, {\rm div} \, f =0 \}$.

%Similarly, we define homogeneous Sobolev spaces $\dot W^s_{p}({\mathbb R}^m)$ and $\dot W^s_{p}(\Om)$, and homogeneous Besov spaces $\dot B^s_{pq}({\mathbb R}^m)$ and $\dot B^s_{pq}(\Om)$.

%For the simplicity, we denote $\dot B^{s}_{pp}(\Omega )=  \dot B^{s }_{p}(\Omega  ), \, B^{s}_{pp}(\Omega  )= B^{s}_{p}(\Omega), \, \dot B^{s }_{pp \sigma}(\Omega )= \dot B^{s }_{p \sigma}(\Omega ), \, B^{s }_{pp\sigma}(\Omega )= B^{s }_{p\sigma}(\Omega  )$.

Now, we introduce  anisotropic  Besov spaces and their properties (See
Chapter 5 of \cite{Triebel2}, and Chapter 3 of \cite{amman-anisotropic} for the definition of   spaces and their properties, although different notations are used in the books).

Let $s\in {\mathbb R}$ and $1\leq p \leq \infty$. We define an
anisotropic homogeneous Sobolev space $\dot W^{s, \frac{s}2}_p ({\mathbb R}^{n+1})$, $ n \geq 1$ by
\begin{eqnarray*}
\dot W^{s, \frac{s}2}_p  ({\mathbb R}^{n+1}) = \{ f \in {\mathcal S}'({\mathbb
R}^{n+1}) \, | \, f  = h_{s} * g, \quad \mbox{for some} \quad g   \in L^p ({\mathbb R}^{n+1}) \}
\end{eqnarray*}
with norm $\|f\|_{\dot W^{s, \frac{s}2}_p  ({\mathbb R}^{n+1})} : = \| g \|_{L^p({\mathbb R}^{n+1})} \,  = \|  h_{-s} * f  \|_{L^p({\mathbb R}^{n+1})},$
 where
$*$ is a convolution in ${\mathbb R}^{n+1}$ and ${\mathcal S}^{'}({\mathbb
R}^{n+1})$
 is the dual space of the Schwartz space
${\mathcal S}({\mathbb R}^{n+1})$.
Here, $h_s $  is a distribution in ${\mathbb R}^{n+1}$ whose Fourier transform
in ${\mathbb R}^{n+1}$ is defined by
\begin{eqnarray*}
{\mathcal F}_{x,t} ( h_{s}) (\xi,\tau) = c_s(  4\pi^2 |\xi|^2 + 2\pi  i
\tau)^{-\frac{s}{2}}, \quad (\xi, \tau) \in {\mathbb R}^n \times {\mathbb R},
\end{eqnarray*}
where ${\mathcal F}_{x,t} = {\mathcal F}$ is the Fourier transform   and ${\mathcal F}^{-1}$ is  the inverse Fourier tranfsorm in $\RN$. respectively.  We alse denote $\widehat{\cdot}$ as the  Fourier transform in $\RN$. 
Here,  the function $(  4\pi^2 |\xi|^2 + i
\tau)^{-\frac{s}{2}}$ is defined  as  $e^{-\frac{s}2 \ln(4\pi^2 |\xi|^2 + i\tau)}$, where the branch line of the logarithm that  is a negative  real line for real arguments is used.

Let $\displaystyle D^\frac12_t f(t) = \frac1{\sqrt{\pi}}  D_t \int_{-\infty}^t \frac{f(s)}{(t -s)^\frac12} ds$ and $D^{k  +\frac12}_tf = D^\frac12_t D^k_tf$. Note that $\widehat{D^\frac12_t f} (\tau) = \big(a_0 + i b_0 sign(\tau) \big) |\tau|^\frac12 \hat{f}(\tau)$ for certain complex numbers $a_0$ and $b_0$ (see Section 2 in \cite{CK}).
Using the multiplier theorem for $L^p({\mathbb R}^{n+1})$ (see  Section 2 in \cite{CK} or Section 6 in \cite{BL}), in  the case  $\alpha=k \in {\mathbb N} \cup \{ 0\}$,  
\begin{align*}
 \|f\|_{\dot W^{k, \frac{k}2}_p  ({\mathbb R}^{n+1})  } & \approx \sum_{|\be| + l = k} \| D_{x}^\be D^{\frac12 l}_{t}   f\|_{L^p ({\mathbb R}^{n+1})}.
\end{align*}
(See Theorem \ref{proofinR}.) 
%where $\displaystyle D^\frac12_t f(t) = D_t \int_{-\infty}^t \frac{f(s)}{(t -s)^\frac12} ds$ and $D^{k  +\frac12}_tf = D^\frac12_t D^k_tf$.
In particular,  $\displaystyle\| f\|_{\dot W^{0,0}_p({\mathbb R}^{n+1})}  = \| f\|_{ L^p({\mathbb R}^{n+1})}$ and 
\begin{align*}
\| f\|_{\dot W^{1, \frac12 }_{p}({\mathbb R}^{n+1}) }
            =  \| D_t^{\frac12 } f\|_{L^p({\mathbb R}^{n+1})}
               + \|D_x f\|_{L^p({\mathbb R}^{n+1})},\\
\| f\|_{\dot W^{2, 1 }_{p}({\mathbb R}^{n+1}) }
            =  \| D_t  f\|_{L^p({\mathbb R}^{n+1})}
               + \|D^2_x f\|_{L^p({\mathbb R}^{n+1})}.
\end{align*}
See  \cite{CK}.

We fix a Schwartz function $\phi\in {\mathcal S} ({\mathbb R}^{n+1})$ satisfying $\hat{\phi}(\xi,\tau) > 0$ on $\frac12 < |\xi| + |\tau|^\frac12 < 2$, $\hat{\phi}(\xi,\tau)=0$ elsewhere, and
$\sum_{j=-\infty}^{\infty} \hat{\phi}(2^{-j}\xi, 2^{-2j}\tau) =1$ for $(\xi,\tau) \neq 0$.
Let
\begin{align*}
\widehat{\phi_j}(\xi,\tau) &:= \widehat{\phi}(2^{-j} \xi, 2^{-2j}\tau) \quad (j = 0, \pm 1, \pm 2 , \cdots).
%\widehat{\psi}(\xi,\tau) &:= 1- \sum_{j=1}^\infty \widehat{\phi} (2^{-j} \xi, 2^{-2j} \tau).
\end{align*}

We define the homogeneous anisotropic Besov space $\dot B^{s, \frac{s}2}_{p,q} ({\mathbb R}^{n+1})$  as
\begin{align*}
\dot B^{s, \frac{s}2}_{p,q} ({\mathbb R}^{n+1}) = \{ f \, | \, \| f\|_{\dot B^{s, \frac{s}2}_{p,q} ({\mathbb R}^{n+1})  } < \infty \}, 
\end{align*}
where
\begin{align*}
\| f\|_{\dot B^{s, \frac{s}2}_{p,q} ({\mathbb R}^{n+1})} & =  (\sum_{ -\infty<  j < \infty}  2^{q s j}\| f* \phi_j \|^q_{L^p ({\mathbb R}^{n+1})} )^\frac1q,
%\| f\|_{B^{s, \frac{s}2}_{p,q} ({\mathbb R}^{n+1})} & = \| f * \psi \|_{L^p ({\mathbb R}^{n+1})} + (\sum_{1 \leq j < \infty}  2^{p s j}\| f* \phi_j \|^q_{L^p ({\mathbb R}^{n+1})} )^\frac1q,
\end{align*}
where $*$ is   the $n +1$ dimensional convolution.

\begin{rem}\label{rem0712}
\begin{itemize}
\item[(1)]
We consider  $\dot W^{s,\frac{s}2}_p (\RN)$ and $\dot B^{s, \frac{s}2}_{p,q} (\RN)$ as the quotient spaces with polynomial space with respect to $(x,t)$ and so $\dot H^s_p (\RN)$ and $\dot B^{s, \frac{s}2}_{p,q} (\RN)$ are normed space.

\item[(2)]
For $ s \in {\mathbb R}$, $1 < p < \infty$, the anisotropic  homogeneous Sobolev space  $\dot W^{s, \frac{s}2}_p (\RN)$ norm is equivalent to 
\begin{align}\label{0712-5}
  \| {\mathcal F}^{-1} \big( \sum_{-\infty< k < \infty} (  4\pi^2 |\xi| +2\pi i \tau)^\frac{s}2  \hat \phi(2^{-k} \xi, 2^{-2k}\tau) \hat{f}\big)\|_{L^p (\RN)}.
\end{align}
(See Appendix \ref{appendix240721})

\item[(3)]
Note that   $\dot W^{s,\frac{s}2}_p (\RN)$ and $\dot B^{s, \frac{s}2}_{p,q} (\RN)  $ are   complete spaces  for  $  0< s< \frac{n+2}p$ (see Theorem \ref{proofinR}).

\end{itemize}
\end{rem}
 
The properties of the anisotropic homogeneous  spaces are comparable with the properties of usual  homogeneous spaces. In particular, the following properties will be used later.
\begin{prop}
\label{prop2}
\begin{itemize}
\item[(1)]
For  $1 \leq  p, \,\, q_0,\,\, q_1,\,\,  r\leq \infty$ and $ s_1, \, s_2 \in {\mathbb R}$,
\begin{align*}
&( \dot W^{ s_0, \frac{s_0}2}_p   ({\mathbb R}^{n+1}),   \dot W^{ s_1, \frac{s_1}2}_p ({\mathbb R}^{n+1} ))_{\te, r} = \dot B^{s, \frac{s}2}_{p,r} ({\mathbb R}^{n+1}),  \\
&  (\dot B^{s_0, \frac{s_0}2 }_{p,q_0}   ({\mathbb R}^{n+1}),   \dot B^{s_1, \frac{s}2 }_{p,q_1} ({\mathbb R}^{n+1} ))_{\te, r} = \dot B^{s, \frac{s}2}_{q,r} ({\mathbb R}^{n+1}),   \\
&[ \dot W^{ s_0, \frac{s_0}2}_p   ({\mathbb R}^{n+1}),   \dot W^{ s_1, \frac{s_1}2}_p ({\mathbb R}^{n+1} )]_{\te} = \dot W^{s, \frac{s}2}_{p} ({\mathbb R}^{n+1}),  
\end{align*}
where $0 < \te<1$ and  $s = s_0  (1 -\te) + s_1 \te$.

\item[(2)]
Suppose that $ 1 \leq p_0\leq p_1  \leq \infty, \,  1 \leq r_0\leq r_1  \leq \infty$, and $ s_0\geq s_1$ with $ s_0 - \frac{n+2}{p_0} =s_1 - \frac{n+2}{p_1}$.
Accordingly, the following inclusions hold
\begin{align*}
\dot  W^{ s_0,  \frac{s_0}2  }_{p_0} ({\mathbb R}^{n+1} ) \subset   \dot W^{s_1, \frac{s_1}2  }_{p_1} ({\mathbb R}^{n+1}), \quad 1 < p, \,\, p_1 < \infty,\\
\dot  B^{ s_0, \frac{ s_0}2  }_{p_0, r_0} ({\mathbb R}^{n+1} ) \subset   \dot B^{s_1,  \frac{s_1}2  }_{p_1, r_1} ({\mathbb R}^{n+1}), \quad 1 \leq p_0, \,\, p_1 \leq \infty.
\end{align*}

\item[(3)]
For $0 < s$ and $ 1  < p  <  \infty$,
\begin{align*}
\dot W^{s,\frac{s}2}_{p} ({\mathbb R}^{n+1})   &=L^p ({\mathbb R}; \dot W^{s}_{p} (\R)) \cap L^p (\R; \dot W^{\frac{s}2}_{p}({\mathbb R})).
%\dot B^{s,\frac{s}2}_{p,q} ({\mathbb R}^{n+2})   &=L^p ({\mathbb R}; \dot B^{s}_{p,q} (\R)) \cap L^p (\R; \dot B^{\frac{s}2}_{p,q}({\mathbb R})).}}
 \end{align*}
 
 \item[(4)]
\ Let $ 1 < p, \, q < \infty$ and $ 0 < s < 1$. The $\dot B_{p,q}^{s, \frac{s}2} ({\mathbb R}^{n+1})$-norm of $f$ is equivalent to 
\begin{align*}
\| f\|_{\dot B^{s,\frac{s}2}_{p,q} (\RN) } & =   \Big(   \int_{{\mathbb R}^{n+1}}  
\frac1{(|y| +|\tau|^\frac12)^{n+2  + ps  } } \big( \int_{{\mathbb R}^{n+1}}  |f(x +y,t+\tau) - f (x,t)|^p  dxdt \big)^{\frac{q}p}  dyd\tau  \Big)^\frac1q.
\end{align*}
 
 \item[(5)]
If   $f \in   \dot B_{p,q}^{s, \frac{s}2} ({\mathbb R}^{n+1})$ or $f \in   \dot W_{p}^{s, \frac{s}2} ({\mathbb R}^{n+1}), \, s > \frac{2}p$, then $f|_{t =0 } \in  \dot  B_{p,q}^{s -\frac{2}p }(\R )$ with
\begin{align*}
\| f|_{t=0} \|_{ \dot B_{p,q}^{s -\frac2p} (\R)} \leq c \| f|\|_{ \dot B^{s, \frac{s}2}_{p,q} ({\mathbb R}^{n+1})}, \quad 
\| f|_{t=0} \|_{ \dot B_{p,p}^{s -\frac2p} (\R)} \leq c \| f|\|_{ \dot W^{s, \frac{s}2}_{p} ({\mathbb R}^{n+1})}.
\end{align*}

\end{itemize}
\end{prop}
\begin{rem}
    We refer to (2) of Remark \ref{rem0712} and Theorem 3.7.1 in \cite{amman-anisotropic} and Theorem 6.4.5 in \cite{BL} for the proof of (1), Theorem 6.5.1 in \cite{BL} for (2),  Theorem 3.7.3 in \cite{amman-anisotropic} for (3), Theorem 3.6.1 in \cite{amman-anisotropic} for (4)  and Theorem 4.5.2 in \cite{amman-anisotropic} and  Theorem 6.6.1 in \cite{BL} for (5).
\end{rem}

  $ \dot {B}^{s, \frac s2 }_{p,q} (\R_+ \times (0, \infty))$ refers to 
 the restriction of $ \dot {B}^{s, \frac s2 }_{p,q} ({\mathbb R}^n\times {\mathbb R})$ with norm
\begin{align*}
\|f\|_{\dot  B^{s, \frac{s}2 }_{p,q} (\R_+ \times (0, \infty))}=\inf\{ \|F\|_{ \dot B^{s, \frac{s}2 }_{p,q}(\R \times {\mathbb R})}: F\in  \dot B^{s, \frac{s}2}_{p,q}(\R \times {\mathbb R})\mbox{ with } F|_{\R_+ \times (0, \infty)}=f\}.
\end{align*}
Similarly, we define  $\dot W^{s,\frac s2 }_{p} (\R_+ \times (0, \infty))$.

Now, we explain interpolation spaces defined in $\R_+ \times (0, \infty)$. 

For $k \in {\mathbb N} \cup \{ 0\}$, we define $\dot {\mathbb W}_p^{2 k, k} (\R_+ \times (0,\infty))$ by
\begin{align*}
\dot {\mathbb W}^{2k, k}_p (\R_+ \times (0, \infty)) = \{ f   \, | \, \| f\|_{\dot {\mathbb W}^{2k, k}_p (\R_+ \times (0, \infty))} : =\sum_{2l + |\be|  =2 k} \| D_t^l D_x^{\be} f\|_{L^p (\R_+ \times (0, \infty))} < \infty \}.
\end{align*}
We also define the interpolation spaces
$ \dot  {\mathbb W}^{s, \frac{s}2}_p(\R_+ \times (0, \infty))$ and $\dot {\mathbb B}^{s, \frac{s}2}_{p,q} (\R_+ \times (0, \infty))$ by
\begin{align*}
\dot {\mathbb W}^{s, \frac{s}2}_{p} (\R_+ \times (0, \infty))& : =  [L^{p}   (\R_+ \times (0, \infty)),   \dot {\mathbb W}^{2k, k}_{p} (\R \times (0, \infty))]_{\te}, \\
\dot {\mathbb B}^{s, \frac{s}2}_{p,q} (\R_+ \times (0, \infty))& : =  (L^{p}   (\R_+ \times (0, \infty)),   \dot {\mathbb W}^{2k, k}_{p} (\R \times (0, \infty)))_{\te,  q},
\end{align*}
where $0 < \te<1$ and  $s = 2 k \te$.

For  $ f  \in \dot {\mathbb W}^{2k, k}_{p} (\R_+ \times (0, \infty)) $,  we define   extension $E f \in   \dot {\mathbb W}^{2k, k}_{p} ({\mathbb R}^{n+1})$  of $f$  by the following process. First, for the given function 
  $ f : {\mathbb R}^{n}_+  \times (0, \infty) \ri  {\mathbb R} $, we define extension of $E_1 f$ over $\R\times (0, \infty)$  as follows:
\begin{align*}
E_1 f (x,t) =\left\{\begin{array}{l}\vspace{2mm}
  f(x,t) \quad \mbox{for} \quad  x_n > 0,\\
  \sum_{j =1}^{2k+1} \la_j f(x',-jx_n,t)\quad \mbox{for} \quad  x_n < 0,
  \end{array}
  \right.
  \end{align*}
  where $\la_1, \, \la_2, \cdots, \la_{2k+1}$ satisfy $\sum_{j =1}^{2k+1} (-j)^l \la_j =1, \,\,  0 \leq l \leq 2k$.  Next,  for  $ f : {\mathbb R}^{n}  \times (0, \infty) \ri  {\mathbb R} $, we define extension of $E_2 f$ over ${\mathbb R}^{n+1} $  by
\begin{align*}
E_2 f (x,t) =\left\{\begin{array}{l}\vspace{2mm}
  f(x,t) \quad \mbox{for} \quad   t> 0,\\
  \sum_{j =1}^{k+1} \eta_j f(x, -j t)\quad \mbox{for} \quad  t < 0,
  \end{array}
  \right.
  \end{align*}
  where $\eta_1, \, \eta_2, \cdots, \eta_{k+1}$ satisfy $\sum_{j =1}^{k+1} (-j)^l \eta_j =1, \,\,  0 \leq l \leq k$.
  Let $ E = E_2 E_1$. Accordingly, for  $ f  \in \dot {\mathbb W}^{2k, k}_{p} (\R_+ \times (0, \infty)) $,   $E f \in \dot W^{2 k, k}_p ({\mathbb R}^{n+1})$ with $\| E f\|_{\dot W^{2 k, k}_p ({\mathbb R}^{n+1})} \leq c \| f \|_{\dot {\mathbb W}^{2k, k}_p (\R_+ \times (0, \infty))}$    (see Theorem 5.19 in \cite{AF}).    This implies that $ \dot {\mathbb W}^{2 k, k}_p (\R_+ \times (0, \infty)) =  \dot W^{2 k, k}_p (\R_+ \times (0, \infty))$ for $k \in {\mathbb N} \cup \{ 0\}$.

 Let $s > 0$ be a non-natural number.  Based on  the property of the complex interpolation,   $E: \dot {\mathbb W}^{s, \frac{s}2}_p (\R_+ \times (0, \infty)) \ri  \dot W^{s, \frac{s}2}_p ( {\mathbb R}^{n+1})$  is bounded operator. This implies  the followings:
\begin{align*}
\| f \|_{\dot  W^{s,\frac{s}2}_p (\R_+ \times (0, \infty))} \leq  c \|E f\|_{\dot  W^{s,\frac{s}2}_p (\RN)} \leq c \| f\|_{ \dot {\mathbb W}^{s,\frac{s}2 }_p (\R_+ \times (0, \infty))}.
\end{align*}
Accordingly, 
\begin{align}\label{subset1}
 \dot {\mathbb W}^{s, \frac{s}2}_p (\R_+ \times (0, \infty)) \subset \dot W^{s, \frac{s}2 }_p (\R_+ \times (0, \infty)).
 \end{align}

Let $2k < s < 2k+2$ for $k \in {\mathbb N} \cup \{0\}$. Let $ R : \dot W^{2i,i}_p ({\mathbb R}^{n+1}) \ri \dot W^{2 i,i}_p (\R_+ \times   (0, \infty))$, $ i = k,,\,k+1$ be a restriction operator defined by $ RF = F|_{\R_+ \times (0, \infty)}$. Clearly, $R$ is bounded operator in $\dot W^{2i,i}_p ({\mathbb R}^{n+1})$ to $\dot W^{2 i,i}_p (\R_+ \times   (0, \infty))$.  By the property of the complex interpolation, $R:\dot W^{s, \frac{s}2}_p ({\mathbb R}^{n+1}) \ri  \dot {\mathbb W}^{s, \frac{s}2}_p (\R_+ \times (0, \infty))$ is bounded operator with 
\begin{align}\label{0501-1}
\| R F \|_{\dot {\mathbb W}^{s, \frac{s}2}_p (\R_+ \times (0, \infty))} \leq c \| F\|_{\dot W^{s, \frac{s}2}_p ({\mathbb R}^{n+1}) }.
\end{align}

Let $ f \in \dot W^{s, \frac{s}2}_p (\R_+ \times (0, \infty)$. By the definition of $  \dot W^{s, \frac{s}2}_p (\R_+ \times (0, \infty)$, there is $F \in \dot W^{s, \frac{s}2}_p ({\mathbb R}^{n+1})$ such that $ RF =   f$ and $\| F\|_{\dot W^{s, \frac{s}2}_p ({\mathbb R}^{n+1})} \leq 2 \| f\|_{ \dot W^{s, \frac{s}2}_p (\R_+ \times (0, \infty)}$.  Then, from \eqref{0501-1}, we have 
\begin{align*}
\| f \|_{\dot {\mathbb W}^{s, \frac{s}2}_p (\R_+ \times (0, \infty))}  = \| R F \|_{\dot {\mathbb W}^{s, \frac{s}2}_p (\R_+ \times (0, \infty))}\leq c \| F\|_{\dot W^{s, \frac{s}2}_p ({\mathbb R}^{n+1}) } \leq 2c
\| f\|_{ \dot W^{s, \frac{s}2}_p (\R_+ \times (0, \infty)}. 
\end{align*}
 This implies the following:
\begin{align}\label{subset3}
 \dot W^{s, \frac{s}2 }_p (\R_+ \times (0, \infty)) \subset \dot {\mathbb W}^{s, \frac{s}2}_p (\R_+ \times (0, \infty)).
 \end{align}
From \eqref{subset1} and \eqref{subset3}, 
\begin{align}\label{230126-1}
\dot W^{s, \frac{s}2 }_p (\R_+ \times (0, \infty)) =  [L^{p}   (\R_+ \times (0, \infty)),   \dot W^{2k, k}_{p} (\R \times (0, \infty))]_{\te}\quad \forall s > 0.
\end{align}
Similarly, we obtain 
\begin{align}\label{230126-2}
  \dot B^{s, \frac{s}2 }_{p, r} (\R_+ \times (0, \infty)) = ( L^p   (\R_+ \times (0, \infty)),   \dot W^{2 k, k}_p (\R_+ \times (0, \infty)))_{\te, r}\quad 1 \leq p,r \leq \infty, \quad  \forall s > 0.
\end{align}
Hereafter, we use the notations $ L^p = L^p(\R_+ \times (0, \infty))$,  $ \dot W^{s, \frac{s}2 }_{p} = \dot W^{s, \frac{s}2 }_{p} (\R_+ \times (0, \infty)) $ and   $\dot B^{s,\frac{ s}2}_{p, r} =  \dot B^{s, \frac{s}2}_{p,r} (\R_+ \times (0, \infty))$.

From Proposition \ref{prop2}, we obtain the following results.
\begin{prop}\label{prop0215} 
\begin{itemize}
\item[(1)]
For  $1 \leq  p, \,\, q_0,\,\, q_1,\,\,  r\leq \infty$,
\begin{align*}
( \dot W^{s_0, \frac{s_0}2}_p,   \dot W^{s_1, \frac{s_1}2}_p )_{\te, r} = \dot B^{s, \frac{s}2}_{p,r}, \qquad
( \dot B^{s_0, \frac{s_0}2}_{p, q_0},   \dot B^{s_1, \frac{s_1}2}_{p,q_1} )_{\te, r} = \dot B^{s, \frac{s}2}_{p,r}, \qquad
[\dot W^{s_0, \frac{s_0}2}_p,   \dot W^{s_1, \frac{s_1}2}_p ]_{\te} = \dot W^{s, \frac{s}2}_{p},
%  (\dot W^{s, \frac{s}2 }_{p_0}),   \dot W^{s, \frac{s}2 }_{p_1})_{\te, r} = \dot B^{s, \frac{s}2}_{q, r},
\end{align*}
where $0 < \te<1$ and  $s = s_0  (1 -\te) + s_1 \te$.

\item[(2)]
Suppose that $ 1 \leq p_0\leq p_1  \leq \infty, \,  1 \leq r_0\leq r_1  \leq \infty$ and $ s_0\geq s_1$ with $s_0 - \frac{n+2}{p_0} = s_1 - \frac{n+2}{p_1}$.
Accordingly, the following inclusions hold
\begin{align*}
\dot  W^{s_0,  \frac{s_0}2  }_{p_0}  \subset   \dot W^{s_1, \frac{s_1}2  }_{p_1}, \quad
\dot  B^{ s_0, \frac{ s_0}2  }_{p_0, r_0} \subset   \dot B^{s_1, \frac{ s_1}2  }_{p_1 ,r_1} .
\end{align*}
In particular, 
\begin{align}
\dot  W^{1,  \frac12  }_{\frac{n+2}2}   \subset   \dot W^{0, 0}_{ n+2} = L^{n+2}.
\end{align}

\item[(3)]
For $0 < s$ and $ 1  < p  <  \infty$,
\begin{align*}
\dot W^{s,\frac{s}2}_{p}   &=L^p (0,\infty; \dot W^{ s}_{p} (\R_+ )) \cap L^p (\R_+; \dot W^{\frac{s}2}_{p}(0, \infty)).
%\dot B^{s,\frac{s}2}_{p,q}   &=L^p (0,\infty; \dot B^{ s}_{p,q} (\R)) \cap L^p (\R; \dot B^{\frac{s}2}_{p,q}(0, \infty)).
 \end{align*}
 
%\item[(4)]
%\ Let $ 1 < p, \, q < \infty$ and $ 0 < s < 1$. The $\dot B_{p,q}^{s, \frac{s}2} $-norm of $f$ is equivalent to
% \begin{align*}
%  \Big( \int_0^\infty \int_{{\mathbb R}^{n}_+} \big(\int_0^\infty\int_{{\mathbb R}^{n}_+} \frac{|f(x,t) - f(y,\tau)|^p   }{(|x-y| +|t -\tau|^\frac12 )^{\frac{p}q(n+2) + ps  }} dyd\tau     \big)^{\frac{q}p}  dxdt  \Big)^\frac1q.
%\end{align*}

 \item[(4)]
If   $f \in   \dot B_{p,q}^{s, \frac{s}2} $ or $f \in   \dot W_{p}^{s, \frac{s}2}, \, \,  1 < p < \infty, \,  \, 1 \leq q \leq \infty, \,\, s > \frac{2}p$, then  $f|_{t =0 } \in  \dot  B_{p,p}^{s -\frac{2}p }(\R )$ or $f|_{t =0 } \in  \dot  B_{p,q}^{s -\frac{2}p }(\R )$, respectively   with
\begin{align*}
\| f|_{t=0} \|_{ \dot B_{p,p}^{ s -\frac{2}p } (\R )} \leq c \| f \|_{\dot  W_{p}^{s, \frac{s}2}}, \quad
\| f|_{t=0} \|_{ \dot B_{p,q}^{ s -\frac{2}p } (\R )} \leq c \| f \|_{\dot  B_{p,q}^{s, \frac{s}2}}.
\end{align*}

\end{itemize}
\end{prop}

From the definition of anisotropic Besov space in ${\mathbb R}^{n+1}$, we obtain the following lemma;
\begin{lemm}\label{tracelemm}
Let  $1 < p < \infty$ and $1 \leq q \leq \infty $ and $ s > \frac1p$. Accordingly, 
\begin{align*}
\| f|_{x_n=0} \|_{ \dot B_{p,p}^{ s -\frac{1}p,\frac{s}2  -\frac1{2p} } (\Rn  \times (0, \infty))} \leq c \| f \|_{\dot  W_{p}^{s, \frac{s}2}}, \quad
\| f|_{x_n =0} \|_{ \dot B_{p,q}^{ s -\frac{1}p, \frac{s}2 -\frac1{2p} } (\Rn \times (0, \infty))} \leq c \| f \|_{\dot  B_{p,q}^{s, \frac{s}2}}.
\end{align*}
\end{lemm}
See Theorem 4.5.2 in \cite{amman-anisotropic}.

%This interpolation theorem used in later.
%\begin{prop}
%Let $X_1$ and $X_2$ be Banach space. Then,
%\begin{align}
%(L^p (0, \infty; X_1), L^p (0, X_2) ) _{\te, q} = L^p (0, \infty; (X_1, X_2)_{\te, q}). 
%\end{align}
%\end{prop}

%Let $I = (0, \infty)$ or ${\mathbb R}$. 
We say that a distribution $f$ in $\Rn \times  {\mathbb R}$ is contained in function space $\dot W^{\frac{k}2}_p(0, \infty;\dot{B}^{-\frac{1}{p}}_{p,p}(\Rn)), \,\, k \in {\mathbb N} \, \cup \, \{ 0 \}$
if $f$ satisfies
\begin{align*}
%\| f\|_{\dot B^s_{p,p}({\mathbb R};\dot{B}^{-\frac{1}{p}}_{p,p}(\Rn))} :& = \Big( \int_{{\mathbb R}} \int_{{\mathbb R}} \frac{ \|  f(\cdot, t) -   f(\cdot, s)\|^p_{\dot B^{-\frac1p}_{p,p} (\Rn)}  }{ |t-s|^{1 +   p s}  }     dsdt \Big)^\frac1p <\infty,\\
\| f\|_{\dot W^{\frac{k}2}_{p}(0, \infty;\dot{B}^{-\frac1p}_{p,p}(\Rn))} :& =  \| D^{\frac{k}2}_t f \|_{L^p (0, \infty; \dot B^{-\frac1p}_{p,p}   (\Rn))} <\infty,
\end{align*}
where $\dot{B}^{-\frac1p}_{p,p}(\Rn)$ is dual space of $\dot{B}^{\frac1p}_{p',p'}(\Rn)$,  $ \frac1p + \frac1{p'} =1$.

Let  $ k< s < k +1$   with $ s =\te k   + ( 1 -\te) (k+1) $ for $ 0 < \te < 1$ and $1 \leq q \leq  \infty$. We define  function space 
$\dot A^{s, \frac{s }2}_{p,q} ( (\Rn \times (0, \infty))$ by 
\begin{align}\label{interpolationspace2}
\begin{split}
  \dot A^{s, \frac{s }2}_{p,q}  (\Rn \times (0, \infty)) : &=
\Big( L^p(0, \infty;\dot {B}^{k-\frac{1}{p}}_{p,p}(\Rn))\, \cap \, \dot W^{\frac{k}2}_{p}(0, \infty;\dot{B}^{-\frac1p}_{p,p}(\Rn)), \\
 &\hspace{2cm}  L^p(0, \infty;\dot {B}^{k+1-\frac{1}{p}}_{p,p}(\Rn))\, \cap \, \dot W^{\frac{k+1}2}_{p}(0, \infty;\dot{B}^{-\frac1p}_{p,p}(\Rn)) \Big)_{\te, q}.
   \end{split}
\end{align}

The following H$\ddot{\rm o}$lder type inequality  is a well-known result  for usual homogeneous Besov space in $\R$ (see Lemma 2.2 in \cite{chae}).
\begin{lemm}\label{0510prop}
Let $0 < s  $, $1 \leq p \leq r_i, p_i \leq \infty$,  $ 1 \leq q \leq \infty$ with  $\frac1{r_i} + \frac1{p_i} = \frac1p$, $i=1,2$. Accordingly,
\begin{align} \label{1104-1}
\| f_1f_2\|_{ \dot B^{s, \frac{s}2}_{p,q}   }  \leq
        c \big( \| f_1\|_{   \dot B^{s, \frac{s}2}_{p_1,q }   } \| f_2\|_{  L^{r_1}  }   + \| f_1\|_{ L^{p_2 } } \|f_2\|_{   \dot B^{s, \frac{s}2}_{r_2 ,q} }  \big).
\end{align}
\end{lemm}
\begin{proof}
Let $E$ be an extension operator defined by the  above statement. From the proof of   Lemma 2.2 in \cite{chae}, Lemma \ref{0510prop} holds for $E f_1$ and $E_2f$ with homogeneous anisotropic Besov space $\dot B^{s, \frac{s}2}_{p,q} ( {\mathbb R}^{n+1}   )$.  Accordingly, we have 
\begin{align*}
\notag \| f_1f_2\|_{\dot B^{s, \frac{s}2}_{p,q}  } & \leq  \| E( f_1) E( f_2 )\|_{  \dot B^{s, \frac{s}2}_{p,q} ( {\mathbb R}^{n+1}   ) }\\
\notag &\leq
        c \big( \| E f_1\|_{   \dot B^{s, \frac{s}2}_{p_1 ,q} ({\mathbb R}^{n+1}  ) } \| E f_2\|_{  L^{r_1} ({\mathbb R}^{n+1}  )}   + \|E f_1\|_{ L^{p_2 }({\mathbb R}^{n+1}  )} \| E f_2\|_{   \dot B^{s, \frac{s}2}_{r_2, q}({\mathbb R}^{n+1}   )}  \big)\\
 & \leq
        c \big( \|   f_1\|_{   \dot B^{s, \frac{s}2}_{p_1,q }   } \|   f_2\|_{  L^{r_1}  }   + \|   f_1\|_{ L^{p_2 } } \|   f_2\|_{   \dot B^{s, \frac{s}2}_{r_2,q } }  \big).
\end{align*}
Hence, \eqref{1104-1} holds.
\end{proof}

\begin{rem}
The properties (1) and (2) of Proposition \ref{prop0215}  and Lemma \ref{0510prop} hold for  Besov space $\dot B^s_{p,q} (\R_+)$ and Sobolev space $\dot W^s_p (\R_+)$.
\end{rem}

The next lemma pertains to the relation between $\dot B^{s, \frac{s}2}_{p,q} $ and $C_b ([0, \infty); \dot  B^{s -\frac{2}p}_{p,q} (\R_+))$.
\begin{lemm}\label{lemma1208}
Let $ 1 < p < \infty$, $ 1 \leq q \leq \infty$ and  $ s > \frac2p$. Accordingly,
\begin{align*}
\dot W^{ s,\frac{s}2}_{p}  \subset  C_b ([0, \infty); \dot  B^{s -\frac{2}p}_{p,p} (\R_+)), \quad
\dot B^{s, \frac{s}2}_{p,q}  \subset  C_b ([0, \infty); \dot  B^{s -\frac{2}p}_{p,q} (\R_+)).
\end{align*}
\end{lemm}
We proved Lemma \ref{lemma1208} in Appendix \ref{prooflemma1208}.

\begin{lemm}\label{lemma0716}
Let $1 \leq p,\, q < \infty$ and $ 0 < s< 1$.
\begin{itemize}
\item[(1)]
For  ${\mathbb G} \in \dot B^{s,\frac{s}2}_{p,q}   \, \cap \,  L^\infty $,
\begin{align*}
\| \si ({\mathbb G}){\mathbb G}\|_{\dot B^{s,\frac{s}2}_{p,q}  } \leq c \| D \si\|_{L^\infty(0, \| {\mathbb G}\|_{L^\infty(\R_+)})}  \| {\mathbb G}\|_{L^\infty(\R_+)} \| {\mathbb G}\|_{\dot B^{s, \frac{s}2}_{p,q} },
\end{align*}
where $\si $ is defined in \eqref{0501-2}.

\item[(2)]
\begin{align*}
 \| \si ({\mathbb G}) -\si({\mathbb H})\|_{\dot B^{s,\frac{s}2}_{p,q} }   \leq c \Big( \|  {\mathbb G} -  {\mathbb H}\|_{\dot B^{s, \frac{s}2}_{p,q}  } + \big(  \| {\mathbb G}\|_{\dot B^{s, \frac{s}2}_{p,q}  } +  \| {\mathbb H}\|_{\dot B^{s, \frac{s}2}_{p,q} }    \big) \big(  \|  {\mathbb G}- {\mathbb H}\|_{L^\infty  }  \big) \Big),
\end{align*}
where  $   c \leq  C \Big(  \|D\si\|_{L^\infty (0, \max(\| {\mathbb G}\|_{L^\infty },  \| {\mathbb H}\|_{L^\infty } )}+ \|D^2\si\|_{L^\infty (0, \max(\| {\mathbb G}\|_{L^\infty },  \| {\mathbb H}\|_{L^\infty } )}  \Big) $.

\end{itemize}
\end{lemm}
We proved Lemma \ref{lemma0716} in Appendix \ref{proofoflemma0716}.

The following is the Gagliado-Nirenberg type inequality in anisotropic spaces.
\begin{lemm}\label{lemma0601}
Let   $ p > n+2$ and $\eta= \frac{n+2}{n+2 +p}$ and $\te = \frac{n+2}p$. Accordingly,
  
\begin{align}
\label{230519-4}\| u \|_{L^p (\R_+ \times (0, \infty))} & \leq c \| u \|_{L^{n+2}(\R_+ \times (0, \infty))}^{\te} \| u\|^{1- \te}_{\dot B_{p,1}^{\frac{n+2}p, \frac{n+2}{2p}} (\R_+ \times (0, \infty))},
\\
\label{230519-5}\| u \|_{\dot B^{\frac{n+2}p, \frac{n+2}{2p}}_{p,1} (\R_+ \times (0, \infty))} &  \leq c \| u \|^{\frac{\te (1 -\eta)}{\te + \eta -\te \eta}}_{L^{n+2} (\R_+ \times (0, \infty))}   
\| u \|^{\frac{\eta}{\te + \eta -\te \eta}}_{\dot B^{1 + \frac{n+2}p, \frac12 + \frac{n+2}{2p}}_{p,1} (\R_+ \times (0, \infty))}. 
\end{align}
\end{lemm}
We proved Lemma \ref{lemma0601} in Appendix   \ref{prooflemma0601}.

We now define the weak solution for the Stokes equations  \eqref{maineq-stokes}. 
\begin{defin}[Weak solution of  the Stokes equations]
\label{stokesdefinition}
Let  $1<p<\infty$ and $1\leq \al\leq 2$.
Let $u_0$, $g$ and $   {\mathcal F}$ satisfy the same hypotheses as in Theorem \ref{thm-stokes}.
A vector field $u\in    B^{\al,\frac{\al}2}_{p,q}(\R_+ \times (0, \infty))$  is called a weak solution of the Stokes equations \eqref{maineq-stokes} if  the following conditions are satisfied:%\\
 \begin{align*}
&\int^\infty_0\int_{\R_+} \big( - u\cdot(  \Phi_t+ \Delta \Phi)   + {\mathcal F}:\nabla \Phi  \big) dxdt\\
& \quad = + \int_0^\infty \int_{\Rn} g(x', t) \cdot \frac{\pa \Phi}{\pa {\bf n}} d x' dt + \int_{\R_+} u_0(x) \cdot \Phi(x,0) dx 
\end{align*}
for each $\Phi\in C^\infty_c(\overline{\R_+}\times [0,\infty))$ with ${\rm div} _x\Phi=0$ and $\Phi|_{x_n =0} =0$. In addition, for each $\Psi\in C^1_c(\R_+)$
\begin{equation}\label{Stokes-bvp-2200}
\int_{\R_+} u(x,t) \cdot \na \Psi(x) dx =0  \quad  \mbox{ for all}
\quad 0 < t< \infty.
\end{equation}
\end{defin}

Based on the same method, we define the weak solution of  the system   \eqref{maineq2}.

\begin{defin}[Weak solution to the non-Newtonian incompressible fluid] Let $1<p <\infty$ and $1 \leq \al \leq 2$ .
Let $u_0$ and $g$ satisfy the same hypothesis as in Theorem \ref{thm-navier}.
A vector field $u\in B^{\al,\frac{\al}2}_{p,q}(\R_+ \times (0, \infty))$ is called a weak solution of the non-Newtonian fluids \eqref{maineq2} if the following variational formulations are satisfied:\\
 \begin{align}\label{weaksolution-NS}
  \int^\infty_0\int_{\R_+} \Big(- u\cdot \Phi_t +  \big( S(Du)   +(u\otimes u) \big):\nabla \Phi \Big) dxdt = \int_{\Rn} u_0 (x) \cdot \Phi(x,0) dx
\end{align}
for each $\Phi\in C^\infty_c(\overline{\R_+}\times [0,\infty))$ with $\mbox{\rm div} _x\Phi=0$ and $\Phi|_{x_n =0} =0$. In addition, for each $\Psi\in C^1_c(\R_+)$, $u$ satisfies  \eqref{Stokes-bvp-2200}.
\end{defin}

\section{\bf Preliminaries.}
\label{preliminary}
\setcounter{equation}{0}
In the sequel, we denote  the fundamental solutions of the Laplace equation and the heat equation by $N$ and $\Gamma$, respectively, 
\begin{align*}
N(x) = \left\{\begin{array}{ll}\vspace{2mm} \displaystyle -\frac{1}{(n-2)\om_n} \frac{1}{|x|^{n-2}},& n \geq 3,\\
\displaystyle \frac1{2\pi} \ln |x|, & n =2,
\end{array}
\right. \qquad
\Gamma(x,t)  = \left\{\begin{array}{ll} \vspace{2mm}
\displaystyle\frac{1}{\sqrt{4\pi t}^n} e^{ -\frac{|x|^2}{4t} }, &   t > 0,\\
0,  & t < 0,
\end{array}
\right.
\end{align*}
where $\om_n$ is the measure of the unit sphere in $\R$.

We define several  integral operators.
\begin{align*}
N*' f(x,t) & = \int_{\Rn} N(x'-y',x_n) f(y',t) dy',\\
N* g(x) &  = \int_{\R_+ } N(x-y) g(y) dy,\\
 N^** g(x)& = \int_{\R_+ } N(x-y^*) g(y) dy \quad y^* = (y', -y_n),\\
 \Ga g(x,t) & = \int_{\R} \Ga(x-y, t) g(y)dy,\\
%  \Ga^* g(x,t)  & = \int_{\R_+} \Ga(x-y^*, t) g(y)dy,\\
{\mathcal U}f  (x,t) &  = \int_{0}^t \int_{\Rn}  D_{x_n} \Ga(x'-y',x_n, t-s) f (y',s) dy'ds.
%{\mathcal U}^*f  (x,t) &  = \int_0^t \int_{\R_+}   \Ga(x-y^*, t-s) f (y,s) dyds,
\end{align*}
%where $y^* = (y', -y_n) $ for $y_n > 0$.

\begin{prop}
\label{proppoisson2}
Let $ 0< \al$ and $k \in {\mathbb N} \cup \{ 0\}$ and     $1< p < \infty$.
If $f\in L^p(0, \infty; \dot{B}^{\al-\frac{1}{p}}_{p,p}(\Rn))\, \cap \,  \dot B^{\frac{\al}{2}}_{p,p}((0, \infty);\dot{B}^{-\frac{1}{p}}_{p,p}(\Rn))$, then,
\begin{align*}
\|\na N*' f\|_{   \dot W^{2k, k }_{p} }&\leq c   \big( \|f\|_{L^p(0, \infty;\dot {B}^{2k-\frac{1}{p}}_{p,p}(\Rn))} + \|f\|_{\dot W_{p}^{ k }(0, \infty;\dot{B}_{p,p}^{-\frac{1}{p}}(\Rn))}  \big), \quad k \in {\mathbb N} \cup \{ 0\},\\
\|\na N*' f\|_{  B^{\al,\frac{\al}{2}}_{p,q}}&\leq c  \|f\|_{ \dot {A}^{\al, \frac{\al}2}_{p,q} },  \quad \al > 0.
\end{align*}
\end{prop}
We proved Proposition \ref{proppoisson2} in Appendix \ref{proofoflemma230518-0}.

\begin{prop}
\label{proppoisson0518}
Let $ g \in \dot W^{k,\frac{k}2} _p  $ or $ g\in \dot B^{s, \frac{s}2}_{p,q} $,  $ 1 < p < \infty, \, 1 \leq q \leq \infty$. Then,  
\begin{align*}
\|\na^2  N* g\|_{ \dot W^{k,\frac{k}2} _p  }, \,\, \|\na^2  N^* * g\|_{ \dot W^{k,\frac{k}2} _p }&\leq c \|g\|_{ \dot W^{k,\frac{k}2} _p  }, \quad   k  \geq 0,\\
\|\na^2 N* g\|_{ \dot B^{s, \frac{s}2}_{p,q} }, \,\, \|\na^2 N^* * g\|_{ \dot B^{s, \frac{s}2}_{p,q} }&\leq c   \|g\|_{\dot B^{s, \frac{s}2}_{p,q} }, \quad    s > 0.
\end{align*}
\end{prop}
We proved Proposition \ref{proppoisson0518} in Appendix  \ref{proofoflemma230518-1}.

\begin{lemm}
\label{lem-T}
 Let $1<p<\infty$, $ 1 \leq q \leq \infty$ and $ 0< \al < 2$. Let $ f \in \dot{B}^{\al-\frac{1}{p},\frac{\al}{2}-\frac{1}{2p}}_{p,q}(\Rn \times (0, \infty))$ with $f|_{t =0} =0$. Accordingly,
\begin{align*}
\| {\mathcal U}f  \|_{\dot{B}^{\al,\frac{\al}{2}}_{p,q}}  \leq c \| f\|_{\dot{B}^{\al-\frac{1}{p},\frac{\al}{2}-\frac{1}{2p}}_{p,q}(\Rn \times (0, \infty))}, \quad 
\| {\mathcal U}f  \|_{\dot{W}^{\al,\frac{\al}{2}}_{p}(\R_+ \times(0, \infty))} \leq c \| f\|_{\dot{B}^{\al-\frac{1}{p},\frac{\al}{2}-\frac{1}{2p}}_{p,p}(\Rn \times  (0, \infty))}.
\end{align*}
\end{lemm}
\begin{proof}
Let $\tilde f$ be a zero extension of $f$ with respect to $t$. Since $f|_{t =0} =0$,  $\| \tilde f\|_{\dot{B}^{\al-\frac{1}{p},\frac{\al}{2}-\frac{1}{2p}}_{p,p}(\Rn \times   {\mathbb R})}
\leq c \| f\|_{\dot{B}^{\al-\frac{1}{p},\frac{\al}{2}-\frac{1}{2p}}_{p,p}(\Rn \times  (0, \infty))}$.  Let $\tilde {\mathcal U} \tilde f  (x,t)   = \int_{-\infty}^t \int_{\Rn}  D_{x_n} \Ga(x'-y',x_n, t-s)\tilde f (y',s) dy'ds$. 
From \cite{lady?}, we obtain the following estimate
\begin{equation}\label{230519-2}
\| {\mathcal U}  f \|_{\dot{W}^{2,1}_p }\leq c \|   \tilde{\mathcal U} \tilde f \|_{\dot{W}^{2,1}_p(\R \times {\mathbb R})}\leq c\| \tilde f\|_{\dot{B}_{p,p}^{2-\frac{1}{p},1-\frac{1}{2p}}(\Rn \times {\mathbb R})}  \leq c\|  f\|_{\dot{B}_{p,p}^{2-\frac{1}{p},1-\frac{1}{2p}}(\Rn \times  (0, \infty))}.
\end{equation}

By the same argument in the proof  of  Lemma 3.6 in \cite{CJ3},  
\begin{align}\label{230519-3}
\|  {\mathcal U}  f \|_{L^p  } \leq c \|  \tilde{\mathcal U} \tilde f \|_{L^p (\R \times {\mathbb R})  } \leq c \|  \tilde f\|_{\dot B^{-\frac1p, -\frac1{2p}}_{p,p} (\Rn \times  {\mathbb R})} \leq c \|   f\|_{\dot B^{-\frac1p, -\frac1{2p}}_{p,p} (\Rn \times   (0, \infty))}, 
\end{align}
where $\dot B^{-\frac1p, -\frac1{2p}}_{p,p}(\Rn \times  {\mathbb R})$ 
  is dual space of $\dot B^{\frac1p, \frac1{2p}}_{ p',p'}(\Rn \times  {\mathbb R})$. Using the properties of complex interpolation and  real interpolation between \eqref{230519-2} and \eqref{230519-3}, we obtain Lemma \ref{lem-T}.
\end{proof}

\section{Proof of Theorem \ref{thm-stokes}}\label{Proof1}
\setcounter{equation}{0}

Because the proofs are similar, we only prove (1) of  Theorem \ref{thm-stokes}.  We will follow the proof of Theorem 1.2 in \cite{CJ3}.

Let $\widetilde {\mathcal F}$ be an extension of $ {\mathcal F}$ over $\R \times (0, \infty)$ such that 
$\| \widetilde {\mathcal F}\|_{\dot B^{\al -1,\frac{\al -1}2}_{p,q} (\R \times (0, \infty))} \leq c \| {\mathcal F}\|_{ \dot B^{\al -1,\frac{\al -1}2}_{p,q} }$ and  $\widetilde {u}_0$ be an extension of $u_0$ with ${\rm div} \, \widetilde{u}_0 =0$ and $\| \widetilde u_0\|_{ \dot B^{\al -\frac2p}_{p,q} (\R \times (0, \infty))} \leq c \| u_0\|_{ \dot B^{\al -\frac2p}_{p,q} (\R_+)}$.
Let 
\begin{align*}
w^1(x,t) & = \int_0^t \int_{\R} \Ga (x-y, t-s) {\mathbb P} \, {\rm div} \, \widetilde {\mathcal F} (y,s) dyds,\\
w^2 (x,t) &  = \int_{\R} \Ga (x-y, t) \widetilde {u}_0 (y) dy, 
\end{align*}
where $({\mathbb P}f )_i = f_i  -\sum_{1 \leq j \leq n} R_i R_j f_j$ is the Helmholtz decomposition operator in $\R$.  The following are well-known results 
\begin{align}
%\| w^1\|_{\dot B^{\al, \frac{\al}2}_{p,q} (\R \times (0, \infty))} &\leq \| \widetilde {\mathcal F}\|_{ \dot B^{\al-1,\frac{\al}2 -\frac12}_{p,q} ({\mathbb R}^{n+1})} \leq c \| {\mathcal F}\|_{  \dot B^{\al-1,\frac{\al}2 -\frac12}_{p,q}},\\
 \| w^1\|_{\dot W^{k, \frac{k}2}_{p} } &\leq c \| w^1\|_{\dot W^{k, \frac{k}2}_{p} (\R \times (0, \infty))} \leq c  \| \widetilde {\mathcal F}\|_{ \dot W^{k -1,\frac{k}2 -\frac12}_{p} ({\mathbb R}^{n} \times (0, \infty))} \leq c \| {\mathcal F}\|_{  \dot W^{k-1, \frac{k}2 -\frac12}_{p}},\\
%\| w^2\|_{\dot B^{\al, \frac{\al}2}_{p,q} (\R \times (0, \infty))} & \leq c \| \widetilde {u}_0\|_{  \dot B^{\al -\frac2p}_{p,q} (\R)} \leq c  \| u_0\|_{  \dot B^{\al -\frac2p}_{p,q} (\R_+)},\\
\| w^2\|_{\dot W^{k, \frac{k}2}_{p} } & \leq c \| w^2\|_{\dot W^{k, \frac{k}2}_{p} (\R \times (0, \infty))} \leq c  \| \widetilde {u}_0\|_{  \dot B^{k -\frac2p}_{p} (\R)} \leq c  \| u_0\|_{  \dot B^{k-\frac2p}_{p} (\R_+)}, \quad k \in {\mathbb N}.
\end{align}
Using the property of real interpolation, we have 
\begin{align} \label{0502-1}
\| w^1\|_{\dot B^{\al, \frac{\al}2}_{p,q}  }  \leq c \| {\mathcal F}\|_{  \dot B^{\al-1,\frac{\al}2 -\frac12}_{p,q}},\quad 
%\| w^1\|_{\dot W^{2k, k}_{p} } &\leq \| w^1\|_{\dot W^{2k, k}_{p} (\R \times (0, \infty))} \leq c \| \widetilde {\mathcal F}\|_{ \dot W^{2k -1,k -\frac12}_{p} ({\mathbb R}^{n} \times (0, \infty))} \leq c \| {\mathcal F}\|_{  \dot W^{2k-1, k -\frac12}_{p}},\\
\| w^2\|_{\dot B^{\al, \frac{\al}2}_{p,q}  }   \leq c  \| u_0\|_{  \dot B^{\al -\frac2p}_{p,q} (\R_+)} \quad 1 < \al.
%\| w^2\|_{\dot W^{2k, k}_{p} (\R \times (0, \infty))} & \leq c \| \widetilde {u}_0\|_{  \dot B^{2k -\frac2p}_{p} (\R)} \leq c  \| u_0\|_{  \dot B^{2k-\frac2p}_{p} (\R_+)}, \quad k \in {\mathbb N}.
\end{align}

By  trace theorems for  usual homogeneous Sobolev space $\dot W^k_p (\R_+)$ (see (4) of Proposition \ref{prop0215} ), for $k  \in {\mathbb N}$, we have 
\begin{align}\label{0506-1}
\| w^i |_{x_n=0} \|_{L^p (0, \infty; \dot B^{k -\frac1p}_{p,p} (\Rn))} \leq     c \| w^i \|_{L^p (0, \infty, \dot  W_{p}^k (\R_+))} \leq c\|w^i\|_{ \dot W_{p}^{2k,k}}.
\end{align}
From Section 5.2 in \cite{CJ3}, we have 
\begin{align}\label{0506-2}
 \|w^i_n|_{x_n=0}\|_{ \dot W_{p}^{k}( 0,\infty;\dot{B}_{p,p}^{-\frac{1}{p}}(\Rn))}\leq c\|w^i\|_{L^p(\R_+;\dot W_p^{k}(0, \infty))} \leq c\|w^i\|_{ \dot W_{p}^{2k,k}}.
\end{align}
From \eqref{0506-1}, \eqref{0506-2},   the definition of $ \dot A_{p,q}^{\al -\frac1p,\frac{\al }{2} -\frac1{2p}} (\Rn \times (0, \infty)) $ and the properties of the real interpolations  (see \eqref{interpolationspace2}),  
\begin{align}
\label{t41}
 \|w^i_n|_{x_n=0}\|_{ \dot A_{p,q}^{\al -\frac1p, \frac{\al }{2} -\frac1{2p}}}\leq c\|w^i\|_{ \dot {B}_{p,q}^{\al, \frac{\al }{2}}},  \quad \al > 1.
\end{align}

Let  $\phi(x,t) = \int_{\Rn} N(x'-y', x_n)  h(y',t)  dy'$, $w^3 = \na \phi $ and $\pi = -\phi_t$, where\\
 $h : =  g_n(y',t) - w^1_n |_{x_n =0} - w^2_n |_{x_n =0}$.  From the compatibility condition \eqref{compatibility}, we have $h(x',0) =0$.

Accordingly, $(w^3, \pi)$ satisfies
\begin{align*}
\left\{\begin{array}{ll}\vspace{2mm}
&w^3_t -\De w^3 +\na \pi =0 \quad \mbox{in} \quad \R_+ \times (0, \infty),\\ \vspace{2mm}
&{\rm div} \, w^3 =0 \quad \mbox{in} \quad \R_+ \times (0, \infty),\\ \vspace{2mm}
&w^3|_{t =0} =  0,\\
& w^3|_{x_n =0} = (R_1' h , \cdots, R_{n-1}'h, h ),
\end{array}
\right.
\end{align*}
where $R_i', \,\, i = 1, \cdots, n-1$ are $n-1$ dimensional Riesz transform. Moreover, from Proposition \ref{proppoisson2}, \eqref{0502-1} and \eqref{t41},  we have
\begin{align*}%\label{0213-1}
\| w^3 \|_{\dot B^{\al, \frac{\al}2}_{p,q} (\R_+ \times (0, \infty))} &  \leq c \| h\|_{  \dot A^{\al -\frac1p, \frac{\al}2 -\frac1{2p}}_{p,q} (\Rn \times (0, \infty))}\\
&  \leq c \big( \| g_n   \|_{  \dot A^{\al -\frac1p, \frac{\al}2 -\frac1{2p}}_{p,q} (\Rn \times (0, \infty))} + \| {\mathcal F}\|_{ \dot B^{\al-1,\frac{\al}2 -\frac12}_{p,q}}   +\| u_0 \|_{\dot B^{\al -\frac2p}_{p,q} (\R_+)}\big) .
\end{align*}

Let $G: = g    - w^1|_{x_n =0} - w^2|_{x_n =0} - w^3|_{x_n =0} $. 
Note that from  the compatibility condition \eqref{compatibility}  $ G|_{t =0} = g|_{t =0} -u_0|_{x_n =0} =0$ and  $G_n =0$. We solve the following equations 
\begin{align*}
\left\{\begin{array}{ll}\vspace{2mm}
&w^4_t -\De w^4 +\na q =0\quad \mbox{in} \quad \R_+ \times (0, \infty),\\ \vspace{2mm}
&{\rm div} \, w^4 =0 \quad \mbox{in} \quad \R_+ \times (0, \infty),\\ \vspace{2mm}
&w^4|_{t =0} =0, \quad w^4|_{x_n =0} = G.
\end{array}
\right.
\end{align*}

According to Section 5 in  \cite{CJ2}, $w^4$ can be rewritten in the following form
\begin{equation}\label{C60K-april7}
w^4_i(x,t) = -{\mathcal U} G_i(x,t)   -4\delta_{in} \sum_{j=1}^{n-1} {\mathcal U}R^{'}_jG_j+ 4
    \frac{\partial}{\partial x_i} {\mathcal S}(x,t), \,\, i=1,\cdots,n,
\end{equation}
where ${\mathcal S}$ is  defined by
\begin{align}\label{C70K-april7}
%{\mathcal U}G_i(x,t)&=\int_{-\infty}^t \int_{\Rn} D_{x_n}\Ga(x'-y', x_n, t-\tau)
%G_i(y', \tau) dy' d\tau,\\
{\mathcal S}(x,t)         & = \int_{{\mathbb R}^n_+} \nabla_y (N(x-y) - N(x-y^*) ) \cdot F(y,t) dy,
\end{align}
where
\begin{equation}\label{C100K-april7}
F_j: = -\frac12{\mathcal U}G_j,\qquad j=1\cdots, n-1,
\quad
F_n: = \sum_{j=1}^{n-1}  {\mathcal U} R'_j G_j(x,t).
\end{equation}

Let $\tilde G $ be a zero extension over $\Rn \times {\mathbb R}$.  Accordingly, from Proposition \ref{proppoisson0518} and Lemma \ref{lem-T}, 
\begin{align}
\label{CK25-april11-1}
\begin{split}
 \| w^4\|_{ \dot B^{\al, \frac{\al}2}_{p,q}} & \leq c
  \|{\mathcal U} \tilde G\|_{ \dot B^{\al, \frac{\al}2}_{p,q} ( \R_+ \times {\mathbb R})}\leq c \|G\|_{\dot B^{\al -\frac1p,  \frac{\al}2- \frac1{2p}}_{p,q}(\Rn \times  (0, \infty))}\\
  & \leq c
  \big( \|g\|_{\dot B^{\al -\frac1p,  \frac{\al}2- \frac1{2p}}_{p,q}(\Rn \times  (0, \infty))} +\|w^1\|_{\dot B^{\al -\frac1p,  \frac{\al}2- \frac1{2p}}_{p,q}(\Rn \times  (0, \infty))}\\
  & \quad 
   + \|w^2\|_{\dot B^{\al -\frac1p,  \frac{\al}2- \frac1{2p}}_{p,q}(\Rn \times  (0, \infty))} + \|w^3\|_{\dot B^{\al -\frac1p,  \frac{\al}2- \frac1{2p}}_{p,q}(\Rn \times  (0, \infty))} \big).
   \end{split}
\end{align}
Note that  $w = w^1 + w^2 + w^3+ w^4$ is solution of equation \ref{maineq-stokes} and summing all estimates, we obtain the esmates \eqref{240718-1-1} and \eqref{240718-2-2}.
 
 Next, we show the uniqueness of solution in $\dot W^{1,\frac12}_p$, $  p < n+2$. Let $p^* =\frac{(n+2)p}{n+2 -p}$. Note that $\dot W^{1,\frac12}_p \subset L^{p^*}$. Suppose that  $(u_1,p_1)$ and $(u_2,p_2)$ are    weak
 solutions of  the Stokes equations \eqref{maineq-stokes} in the class $ L^{p^*} $ with the same data,
 then $u_1-u_2$ satisfies the variational formulation
\[
\int^\infty_0\int_{{\mathbb R}^n_+}(u_1-u_2)\cdot(-\phi_t-\Delta \phi+\nabla
\pi)dxdt=0\] for any $\phi \in C_0^\infty (\overline{{\mathbb R}^n_+} \times
[0,\infty))$ with $  \mbox{div}_x \phi =0,\ \phi|_{x_n=0}=0 \quad \mbox{for all} \quad t
\in (0,\infty)  $. Since $\{-\phi_t-\Delta \phi+\nabla \pi:\
\phi \in C_0^\infty (\overline{{\mathbb R}^n_+} \times
[0,\infty))\mbox{ with }  \mbox{div } \phi(\cdot, t) =0,\ \phi|_{x_n=0}=0
 \}
$ is dense in
$L^{(q^*)'} $, we conclude that
$u_1-u_2=0$ a.e. in ${\mathbb R}^n_+\times (0,\infty)$. Therefore, the
uniqueness of the solution of the Stokes system \eqref{maineq-stokes} holds
in the  class $ L^{p^*} $.

\section{Proof of  Theorem \ref{thm-navier}}

\label{nonlinear}
\setcounter{equation}{0}

In this section, we prove Theorem \ref{thm-navier} by  constructing approximate velocities and 
deriving  their uniform convergence in  $ \dot B^{\al,\frac\al 2 }_{p,q } \cap \dot B^{1+\frac{n+2}p, \frac12 + \frac{n+2}{2p}}_{p,1} \cap \dot W^{1,\frac12}_{\frac{n+2}2}$.

\subsection{Approximating solutions}

Let $u^0 =0$ and $p^0=0$.   After obtaining $(u^1,p^1),\cdots, (u^m,p^m)$ construct $(u^{m+1}, p^{m+1})$ which satisfies the following equations
\begin{align}
\label{maineq5}
\begin{array}{l}\vspace{2mm}
u^{m+1}_t - \De u^{m+1} + \na p^{m+1} ={\rm div} \, f^m, \qquad \mbox{div } u^{m+1} =0, \mbox{ in }
 \R_+ \times (0,\infty),\\
\hspace{30mm}u^{m+1}|_{t=0}= u_0, \quad u^{m+1}|_{x_n =0} =g,
\end{array}
\end{align}
where $f^m= \si(Du^m) D u^m - u^m\otimes u^m$ for $m \geq 1$.
%The solution $u^{m+1}$ is represented by
%\begin{align*}
%u^{m+1}(x,t) = \Ga_t* u_0(x) + \int_0^t \na \Ga_{t-s}* (  u^m(s) \otimes u^m(s)) ds + \int_0^t \na \Ga_{t-s}*  f(|Du^m(s)|) D u^m(s) ds.
%\end{align*}

From  Theorem \ref{thm-stokes}, we have
\begin{align}\label{u1}
\begin{split}
 &\| u^1\|_{\dot W^{1, \frac12}_{\frac{n+2}2}}  \leq c   M_{01},\quad 
\| u^1\|_{\dot B^{1 +\frac{n+2}p, \frac12 + \frac{n+2}{2p} }_{p,1}}  \leq c  M_{02}, \quad 
 \| u^1\|_{ \dot B^{\al,\frac\al 2 }_{p,q}}  \leq c    M_{03},
\end{split}
\end{align}
where 
\begin{align*}
M_{01}:  & = \|u_0\|_{  \dot   B^{1-\frac{4}{n+2}}_{\frac{n+2}2, \frac{n+2}2}({\mathbb
R}^{n}_+)} + \| g\|_{\dot B^{ 1-\frac{2}{n+2},\frac{1}2-\frac{1}{n+2}}_{\frac{n+2}2 , \frac{n+2}2}({\mathbb R}^{n-1} \times (0, \infty)) } + \| g_n\|_{A^{1-\frac2{n+2}, \frac12 -\frac1{n+2} }_{\frac{n+2}2,\frac{n+2}2} ( \Rn \times (0, \infty) ) },   \\
M_{02}:  & =  \| u_0\|_{\dot B^{1 +\frac{n}p }_{p,1} (\R_+) } + \| g\|_{ \dot  B^{ 1 + \frac{n+1}p,\frac12 +\frac{n +1 }{2p}}_{p,1}({\mathbb R}^{n-1} \times (0, \infty)) }     + \| g_n \|_{ \dot  A^{ 1 +\frac{n+1}p, \frac12 +\frac{n+1}{2p}}_{p,1} ( \Rn \times (0, \infty)  )  },\\
M_{03}:  & =  \|u_0\|_{  \dot   B^{\al-\frac{2}{p}}_{p,q}({\mathbb
R}^{n}_+)} + \| g\|_{ \dot B^{ \al-\frac{1}{p},\frac{\al} 2-\frac{1}{2p}}_{p,q}({\mathbb R}^{n-1} \times (0, \infty))  }  + \| g_n \|_{\dot A^{\al -\frac1p, \frac{\al}2 -\frac1{2p} }_{p,q} ( \Rn \times (0, \infty))  }.
\end{align*}

%Using Sobolev inequality (see (2) of  Proposition \ref{prop2}) and \eqref{u1}, we have
%\begin{align}\label{0624-3-1}
%\| u^1\|_{L^{n+2}}
%\leq \| u^1\|_{\dot H^{1, \frac12}_{\frac{n+2}2}}\leq c M_{02}.
%\end{align}

Applying  Theorem \ref{thm-stokes},  we have 
\begin{align}
\label{0527-1-1}
  \|  u^{m+1} \|_{ \dot W^{1, \frac12}_{\frac{n+2}2} } & \leq c \big(M_{01}+       \|u^m \otimes u^m\|_{ L^{\frac{n+2}2} }  +   \| \si (Du^m) D u^m\|_{ L^{\frac{n+2}2} } \big),\\
\label{eq1109-2}  \|  u^{m+1} \|_{  \dot B^{1 + \frac{n+2}p, \frac12 + \frac{n+2}{2p}}_{p,1}} & \leq c \big(  M_{02}
      +       \|u^m \otimes u^m\|_{  \dot B^{\frac{n+2}p, \frac{n+2}{2p} }_{p,1}}  +   \| \si (Du^m) D u^m\|_{ \dot B^{\frac{n+2}p,\frac{n+2}{2p} }_{p,1}} \big),\\
 \label{0624-1} \|  u^{m+1} \|_{  \dot B^{\al,\frac{\al}2}_{p,q}} & \leq c \big(M_{03}+       \|u^m \otimes u^m\|_{  \dot B^{\al -1,\frac{\al}2 -\frac12}_{p,q} }  +   \| \si (Du^m) D u^m\|_{ \dot B^{\al -1,\frac{\al}2 -\frac12}_{p,q}} \big).
\end{align}

From H$\ddot{\rm o}$lder inequality and (2) in Proposition \ref{prop0215},  
\begin{align}\label{0601-5}
 \| u^m \otimes u^m\|_{L^{\frac{n+2}2}  } & \leq c \| u^m \|^2_{L^{ n+2}  } \leq c \| u^m \|^2_{\dot H^{1, \frac12}_{\frac{n+2}2}  }.
\end{align}

From the  definition of anisotropic Besov space, we get $\|  D u^m\|_{L^\infty} \leq c \| u^m \|_{\dot B^{1 + \frac{n+2}p,  \frac12 + \frac{n+2}{2p}}_{p,1}}$.  Note that $|\si (Du^m)| = |{\mathbb F}(Du^m)  - {\mathbb F}(0) |\leq c |Du^m|$. Then, from Lemma \ref{lemma0716},   if $\| u^m   \|_{\L^\infty }  \leq 1$,   we have
\begin{align}\label{0601-6}
\begin{split}
 \| \si (Du^m) D u^m \|_{L^{\frac{n+2}2}  } \leq    c\|  D u^m \|_{L^\infty  } \|  D u^m \|_{L^{\frac{n+2}2}  }
 \leq     c\|  D u^m \|_{\dot B^{1 +\frac{n+2}p, \frac12 + \frac{n+2}{2p}  }_{p,1}} \|   u^m \|_{\dot W^{1,\frac12}_{\frac{n+2}2}  }.
\end{split}
\end{align}
From \eqref{0527-1-1},  \eqref{0601-5} and \eqref{0601-6}, 
\begin{align}\label{0601-7}
    \|   u^{m+1} \|_{\dot W^{1,\frac12}_{\frac{n+2}2}  }  \leq c \big(  M_{01}+       \|   u^m \|^2_{\dot W^{1,\frac12}_{\frac{n+2}2}  }   +\|   u^m \|_{\dot B^{1 +\frac{n+2}p, \frac12 + \frac{n+2}{2p}  }_{p,1}} \|   u^m \|_{\dot W^{1,\frac12}_{\frac{n+2}2}  }\big).
\end{align}

From  Lemma \ref{0510prop} and Lemma \ref{lemma0601},   we have 
\begin{align} \label{eq1109-1}
\begin{split}
 \|u^m \otimes u^m\|_{ \dot B^{\frac{n+2}p,\frac{n+2}{2p} }_{p,1} } &\leq c \| u^m \|_{L^\infty} \|u^m \|_{\dot B^{\frac{n+2}p,\frac{n+2}{2p} }_{p,1}  }\leq c  
\| u^m \|^2_{\dot B^{ \frac{n+2}p,  \frac{n+2}{2p}}_{p,1} }\\
& \leq c \|   u^m \|_{\dot W^{1,\frac12}_{\frac{n+2}2}}^{\frac{2\te (1-\eta)}{\te +\eta -\te \eta}}
\|   u^m \|^{\frac{2\eta}{\te +\eta -\te \eta}}_{\dot B^{1 + \frac{n+2}p, \frac12 + \frac{n+2}{2p}}_{p,1} }\\
& \leq c \Big( \|   u^m \|_{\dot W^{1,\frac12}_{\frac{n+2}2}}^2 + 
\|   u^m \|^2_{\dot B^{1 + \frac{n+2}p, \frac12 + \frac{n+2}{2p}}_{p,1} } \Big).
\end{split}
\end{align}
From Lemma \ref{lemma0716},  if $ \| D u^m   \|_{L^\infty}   \leq 1$,  then
\begin{align} \label{eq1109-3}
 \| \si (Du^m) D u^m \|_{\dot B^{\frac{n+2}p,\frac{n+2}{2p} }_{p,1} } & \leq  c  \| D u^m \|_{L^\infty} \| D u^m \|_{\dot B^{\frac{n+2}p,\frac{n+2}{2p}}_{p,1}(\R_+ \times (0, \infty ))}
\leq  c   \|  u^m  \|^2_{\dot B^{1 + \frac{n+2}p, \frac12 +  \frac{n+2}{2p}}_{p,1} }.
\end{align}
From \eqref{eq1109-2},    \eqref{eq1109-1} and \eqref{eq1109-3}, we have   
\begin{align}\label{230519-7}
\| u^{m+1} \|_{  \dot B^{1 + \frac{n+2}p, \frac12 + \frac{n+2}{2p}}_{p,1}} \leq c \big( M_{02} + \|   u^m \|_{\dot W^{1,\frac12}_{\frac{n+2}2}}^2 + 
\|   u^m \|^2_{\dot B^{1 + \frac{n+2}p, \frac12 + \frac{n+2}{2p}}_{p,1} }  \big).
\end{align}
From (1) of Proposition \ref{prop2}, we have 
 $  \dot B^{\al-1, \frac{\al-1}{2}}_{p,q} = (\dot B^{ \frac{n+2}p,  \frac{n+2}{2p}}_{p,q} , \dot B^{\al, \frac{\al}{2}}_{p,q})_{\eta_1, q}$ for $\eta_1 = \frac{\al -1 -\frac{n+2}p}{\al -\frac{n+2}p}  $  and 
$ \dot B^{ \frac{n+2}p,  \frac{n+2}{2p}}_{p,1} \subset \dot B^{ \frac{n+2}p,  \frac{n+2}{2p}}_{p,q}$.  Then, from   Lemma \ref{lemma0601}, we have 
\begin{align*}
\| u^m \|_{\dot B^{\al-1, \frac{\al-1}{2}}_{p,q} } \leq c \| u^m \|^{1 -\eta_1}_{\dot B^{ \frac{n+2}p,  \frac{n+2}{2p}}_{p,1} }\| u^m \|^{\eta_1}_{\dot B^{\al, \frac{\al}{2}}_{p,q} } \leq c \big( \| u^m \|_{\dot B^{ \frac{n+2}p,  \frac{n+2}{2p}}_{p,1} } + 
\| u^m \|_{\dot B^{\al, \frac{\al}{2}}_{p,q} } \big)\\
\leq c \big(\|   u^m \|_{\dot W^{1,\frac12}_{\frac{n+2}2}}+
\|   u^m \|_{\dot B^{1 + \frac{n+2}p, \frac12 + \frac{n+2}{2p}}_{p,1} } + \| u^m \|_{\dot B^{\al, \frac{\al}{2}}_{p,q} }   \big).
\end{align*}
Hence, from  Lemma \ref{0510prop} and Lemma \ref{lemma0601},     we have 
\begin{align}\label{0624-4}
\begin{split}
  \|  u^{m} \otimes u^m \|_{  \dot B^{\al-1,\frac{\al-1}2}_{p,q}}  &\leq c \| u^m \|_{L^\infty}\| u^m \|_{\dot B^{\al-1, \frac{\al-1}{2}}_{p,q} } \leq c  \| u^m \|_{\dot B^{ \frac{n+2}p,  \frac{n+2}{2p}}_{p,1} }   \| u^m \|_{B^{\al-1, \frac{\al-1}{2}}_{p,q} }\\
&   \leq   c \big(\|   u^m \|^2_{\dot W^{1,\frac12}_{\frac{n+2}2}}+
\|   u^m \|^2_{\dot B^{1 + \frac{n+2}p, \frac12 + \frac{n+2}{2p}}_{p,1} } + \| u^m \|^2_{\dot B^{\al, \frac{\al}{2}}_{p,q} }   \big)
\end{split}
\end{align}
and since $\| Du^m \|_{L^\infty} \leq c  \| u^m   \|_{\dot B^{1 + \frac{n+2}p, \frac12 + \frac{n+2}{2p}}_{p,1}} $ from Lemma \ref{lemma0716}, if $ \| u^m   \|_{L^\infty}   \leq  1$, then  we have
\begin{align}\label{0624-2}
 \| \si (Du^m) D u^m \|_{\dot B^{\al -1,\frac{\al -1}2 }_{p,q} } & \leq  c   \| D u^m \|_{L^\infty}  \| D u^m \|_{\dot B^{\al -1,\frac{\al -1}2 }_{p,q} }
\leq  c   \|  u^m  \|_{ \dot B^{1 +\frac{n+2}2, \frac12 +\frac{n+2}{2p}}_{p,1}}\|  u^m  \|_{\dot B^{\al, \frac{\al}2}_{p,q } }.
\end{align}
 
From \eqref{0624-1}, \eqref{0624-4},  and \eqref{0624-2},  
\begin{align}\label{0521-2}
  \|  u^{m+1} \|_{  \dot B^{\al,\frac{\al}2}_{p,q} } & \leq c \big( M_{03} + \|   u^m \|^2_{\dot W^{1,\frac12}_{\frac{n+2}2}}+
\|   u^m \|^2_{\dot B^{1 + \frac{n+2}p, \frac12 + \frac{n+2}{2p}}_{p,1} } + \| u^m \|^2_{\dot B^{\al, \frac{\al}{2}}_{p,q} } \big).
\end{align}

We take   $\de_0 $  satisfying  $\de_0  < \min( \frac1{ 5 c   },    1)$. Let  
\begin{align*} 
 cM_{01},\quad 
   cM_{02}, \quad cM_{03} <  \frac13 \de_0.
\end{align*}
From \eqref{u1},   we have 
\begin{align*} 
 \| u^1\|_{\dot W^{1, \frac12}_{\frac{n+2}2}}, \,\, 
\| u^1\|_{\dot B^{1 +\frac{n+2}p, \frac12 + \frac{n+2}{2p} }_{p,1}}, \,\,  
\| u^1\|_{ \dot B^{\al,\frac\al 2 }_{p,q}} &   < \frac13 \de_0 < \de_0.
\end{align*}

Suppose that 
\begin{align*} 
 \| u^m\|_{\dot W^{1, \frac12}_{\frac{n+2}2}}, \quad  
\| u^m\|_{\dot B^{1 +\frac{n+2}p, \frac12 + \frac{n+2}{2p} }_{p,1}}, \quad 
\| u^m\|_{ \dot B^{\al,\frac\al 2 }_{p,q}}  < \de_0.
\end{align*}
  
Accordingly, from    \eqref{0601-7}, \eqref{230519-7}, and \eqref{0521-2},  
\begin{align*} 
    \|   u^{m+1} \|_{\dot W^{1,\frac12}_{\frac{n+2}2}  }  &  \leq c \big(  M_{01}  +       \|   u^m \|^2_{\dot W^{1,\frac12}_{\frac{n+2}2}  }   +\|   u^m \|_{\dot B^{1 +\frac{n+2}p, \frac12 + \frac{n+2}{2p}  }_{p,1}} \|   u^m \|_{\dot W^{1,\frac12}_{\frac{n+2}2} }\big)\\
& \leq  \big(  \frac13 \de_0   +      c\de_0^2  + c \de_0^2  \big) <\de_0,
\end{align*}   

\begin{align*}
\| u^{m+1} \|_{  \dot B^{1 + \frac{n+2}p, \frac12 + \frac{n+2}{2p}}_{p,1}}  & \leq c \big( M_{02} + \|   u^m \|_{\dot W^{1,\frac12}_{\frac{n+2}2}}^2 + 
\|   u^m \|^2_{\dot B^{1 + \frac{n+2}p, \frac12 + \frac{n+2}{2p}}_{p,1} }  \big)\\
& \leq  \big( \frac13 \de_0  + c\de_0^2    +  c \de_0^2  \big) < \de_0,
\end{align*}
   
\begin{align*} 
  \|  u^{m+1} \|_{  \dot B^{\al,\frac{\al}2}_{p,q} } &      \leq c \big( M_{03} + \|   u^m \|^2_{\dot W^{1,\frac12}_{\frac{n+2}2}}+
\|   u^m \|^2_{\dot B^{1 + \frac{n+2}p, \frac12 + \frac{n+2}{2p}}_{p,1} } + \| u^m \|^2_{\dot B^{\al, \frac{\al}{2}}_{p,q} } \big)\\
& \leq  \big( \frac13 \de_0  +c\de_0^2  +c \de_0^2  \big)  < \de_0.
\end{align*}

Hence, for all $m \geq 1$,  
\begin{align} \label{230524-1}
    \|   u^{m} \|_{\dot W^{1,\frac12}_{\frac{n+2}2}  }, \quad \| u^{m} \|_{  \dot B^{1 + \frac{n+2}p, \frac12 + \frac{n+2}{2p}}_{p,1}}  &  ,\quad 
  \|  u^{m} \|_{  \dot B^{\al,\frac{\al}2}_{p,q} }    <\de_0.
\end{align}

\subsection{Uniform convergence}\label{0714-1}

Let $U^m=u^{m+1}-u^m$ and $P^m=p^{m+1}-p^m$.
Accordingly, $(U^m,P^m)$ satisfies the following equations
\begin{align*}
U^m_t - \De U^m + \na P^m &  =-\mbox{div}(u^m\otimes u^{m}-u^{m-1}\otimes u^{m-1})\\
 & \qquad \  -\mbox{div}(  \si (D u^m) Du^{m} -  \si (D u^{m-1}) D u^{m-1}),\\
  \hspace{30mm} \mbox{div } U^{m}& =0, \mbox{ in }
 \R_+ \times (0,\infty),\\
\hspace{30mm}U^{m}|_{t=0} & = 0, \quad U^m|_{x_n =0} =0.
\end{align*}

%Applying  Theorem \ref{thm-stokes},  we have 

%\begin{align}
%\label{0527-1-1-1}
%  \|  U^{m} \|_{ \dot W^{1, \frac12}_{\frac{n+2}2} } & \leq c \big(    \|u^m \otimes u^m - u^{m-1} \otimes u^{m-1}\|_{ L^{\frac{n+2}2} }  +   \| \si (Du^m) D u^m - \si (Du^{m-1}) D u^{m-1}\|_{ L^{\frac{n+2}2} } \big),\\
%\label{eq1109-2-2}  \|  U^{m} \|_{  \dot B^{1 + \frac{n+2}p, \frac12 + \frac{n+2}{2p}}_{p,1}} & \leq c \big(          \|u^m \otimes u^m - u^{m-1} \otimes u^{m-1}\|_{  \dot B^{\frac{n+2}p, \frac{n+2}{2p} }_{p,1}}  +   \| \si (Du^m) D u^m - \si (Du^{m-1}) D u^{m-1}\|_{ \dot B^{\frac{n+2}p,\frac{n+2}{2p} }_{p,1}} \big),\\
 %\label{0624-1-1} \|  U^{m} \|_{  \dot B^{\al,\frac{\al}2}_{p,q}} & \leq c \big(  \|u^m \otimes u^m- u^{m-1} \otimes u^{m-1}\|_{  \dot B^{\al -1,\frac{\al}2 -\frac12}_{p,q} }  +   \| \si (Du^m) D u^m - \si (Du^{m-1}) D u^{m-1}\|_{ \dot B^{\al -1,\frac{\al}2 -\frac12}_{p,q}} \big).
%\end{align}
Note that 
\begin{align*}
 u^m \otimes u^m- u^{m-1} \otimes u^{m-1} &  = u^m \otimes U^{m-1} + U^{m-1}\otimes u^{m-1},\\
 \si (Du^m) D u^m - \si (Du^{m-1}) D u^{m-1} &  = \si (D u^m) DU^{m-1} + \big( \si (D u^m) -\si (D u^{m-1}) \big) D u^{m-1}.
\end{align*}

From H$\ddot{\rm o}$lder inequality,  (2) in Proposition \ref{prop0215}, we have
\begin{align*}
 \| u^m\otimes u^{m}-u^{m-1}\otimes u^{m-1} \|_{  L^{\frac{n+2}2} } 
 & \leq c \Big(\| u^m\|_{L^{n+2}  } \| U^{m-1}\|_{  L^{n+2}   } + \| u^{m-1}\|_{ L^{n+2}   } \| U^{m-1}\|_{ L^{n+2}   } \Big)\\
& \leq c \Big( \| u^{m}\|_{ \dot W^{1, \frac12}_{\frac{n+2}2}   } + \| u^{m-1}\|_{ \dot W^{1, \frac12}_{\frac{n+2}2}   }  \Big)     \| U^{m-1}\|_{ \dot W^{1, \frac12}_{\frac{n+2}2}   }.
\end{align*}
Note that since $\si (0) =0$, by mean-value theorem,  we have  $ | \si (Du^m)| \leq  \| D \si \|_{L^\infty (0,  \| D u^m\|_{L^\infty})}  | D u^m|$, and $| \si (D u^m) - \si (D u^{m-1})| \leq  \| D \si \|_{L^\infty (0,  \max ( \| D u^m\|_{L^\infty}, \| D u^{m-1}\|_{L^\infty}    ))}  | D u^m - D u^{m-1}|$. Hence, we have 
\begin{align*}
\| \si (D u^m) Du^{m} -  \si (D u^{m-1}) D u^{m-1}\|_{L^{\frac{n+2}2}  }
&\leq c  \big(\| D u^m\|_{L^\infty}  + \| D u^{m-1}\|_{L^\infty}   \big) \|  U^{m-1}\|_{ L^{n+2} } \\
&   \leq c  \big(\|  u^m\|_{\dot B^{1 + \frac{n+2}p, \frac12 + \frac{n+2}{2p}}_{p,1}  }  + \|   u^{m-1}\|_{\dot B^{1 + \frac{n+2}p, \frac12 + \frac{n+2}{2p}}_{p,1}  }   \big) \|  U^{m-1}\|_{ \dot W^{1, \frac12}_{\frac{n+2}2 } }.
\end{align*}

%\begin{align} \label{eq1109-3}
% \|\big( \si (Du^m) -\si(Du^{m-1} \big)  D u^m \|_{\dot B^{\frac{n+2}p,\frac{n+2}{2p} }_{p,1} } & \leq  c   \| \si(D u^m ) -\si (Du^{m-1}) \|_{L^\infty}  \| D u^m \|_{\dot B^{\frac{n+2}p,\frac{n+2}{2p}}_{p,1} }\\
% & \quad  + \| \si(D u^m ) -\si (Du^{m-1}) \|_{ \dot B^{\frac{n+2}p,\frac{n+2}{2p}}_{p,1}}  \| D u^m \|_{  L^\infty}\\
%& \leq  c  \|  u^m  \|_{\dot B^{1 + \frac{n+2}p, \frac12 +  \frac{n+2}{2p}}_{p,1} } \|  U^m  \|_{\dot B^{1 + \frac{n+2}p, \frac12 +  \frac{n+2}{2p}}_{p,1} }.
%\end{align}
%From \eqref{eq1109-2},    \eqref{eq1109-1} and \eqref{eq1109-3}, we have   
%\begin{align}\label{230519-7}
%\| u^{m+1} \|_{  \dot B^{1 + \frac{n+2}p, \frac12 + \frac{n+2}{2p}}_{p,1}} \leq c \big( M_{02} + \|   u^m \|_{\dot W^{1,\frac12}_{\frac{n+2}2}}^{\frac{2\te (1-\eta)}{\te +\eta -\te \eta}}
%\|   u^m \|^{\frac{2\eta}{\te +\eta -\te \eta}}_{\dot B^{1 + \frac{n+2}p, \frac12 + \frac{n+2}{2p}}_{p,1} } + \ep   \|  u^m  \|_{\dot B^{1 + \frac{n+2}p, \frac12 +  \frac{n+2}{2p}}_{p,1} }\big).
%\end{align}

%From (1) of Proposition \ref{prop2}, we have 
 %$  \dot B^{\al-1, \frac{\al-1}{2}}_{p,q} = (\dot B^{ \frac{n+2}p,  \frac{n+2}{2p}}_{p,q} , \dot B^{\al, \frac{\al}{2}}_{p,q})_{\eta_1, q}$ for $\eta_1 = \frac{\al -1 -\frac{n+2}p}{\al -\frac{n+2}p}  $  and 
%$ \dot B^{ \frac{n+2}p,  \frac{n+2}{2p}}_{p,1} \subset \dot B^{ \frac{n+2}p,  \frac{n+2}{2p}}_{p,q}$.  
From   Lemma \ref{lemma0601}, we have 
\begin{align}
\label{240717-1}
\| u^m \|_{\dot B^{\frac{n+2}p, \frac{n+2}{2p}}_{p,1}  } &  \leq c \| u^m \|^{\frac{\te (1 -\eta)}{\te + \eta -\te \eta}}_{L^{n+2} }   
\| u^m \|^{\frac{\eta}{\te + \eta -\te \eta}}_{\dot B^{1 + \frac{n+2}p, \frac12 + \frac{n+2}{2p}}_{p,1}  }
\leq c  \Big( \| u^m \|_{\dot W^{1, \frac12}_{\frac{n+2}2} }    + 
\| u^m \|_{\dot B^{1 + \frac{n+2}p, \frac12 + \frac{n+2}{2p}}_{p,1}  } \Big),\\
\label{240717-3}
\| u^m\|_{L^\infty} & \leq c \| u^m \|_{\dot B^{\frac{n+2}p, \frac{n+2}{2p}}_{p,1}  } 
\leq c  \Big( \| u^m \|_{\dot W^{1, \frac12}_{\frac{n+2}2} }    + 
\| u^m \|_{\dot B^{1 + \frac{n+2}p, \frac12 + \frac{n+2}{2p}}_{p,1}  } \Big),\\
 \begin{split} \label{240717-2}
 \| u^m \|_{\dot B^{\al-1, \frac{\al-1}{2}}_{p,q} } &\leq c \| u^m \|^{1 -\eta_1}_{\dot B^{ \frac{n+2}p,  \frac{n+2}{2p}}_{p,q} }\| u^m \|^{\eta_1}_{\dot B^{\al, \frac{\al}{2}}_{p,q} } \leq c \Big( \| u^m \|_{\dot B^{ \frac{n+2}p,  \frac{n+2}{2p}}_{p,q} } + 
\| u^m \|_{\dot B^{\al, \frac{\al}{2}}_{p,q} } \Big)\\
& \leq   c\Big(\| u^m \|_{\dot W^{1, \frac12}_{\frac{n+2}2} }    + 
\| u^m \|_{\dot B^{1 + \frac{n+2}p, \frac12 + \frac{n+2}{2p}}_{p,1}  } + 
\| u^m \|_{\dot B^{\al, \frac{\al}{2}}_{p,q} } \Big).
\end{split}
\end{align}
For the first inequality of the third equation, we used the fact $ \dot B^{ \frac{n+2}p,  \frac{n+2}{2p}}_{p,1} \subset \dot B^{ \frac{n+2}p,  \frac{n+2}{2p}}_{p,q}$.
 
From  \eqref{240717-1},  \eqref{240717-3} and \eqref{240717-2},     we have 
\begin{align} \label{eq1109-1-1}
\begin{split}
 \|u^m\otimes U^{m-1}  \|_{ \dot B^{\frac{n+2}p,\frac{n+2}{2p} }_{p,1} } &\leq c \Big(  \| u^m \|_{L^\infty} \|U^{m-1} \|_{\dot B^{\frac{n+2}p,\frac{n+2}{2p} }_{p,1}  } + \| U^{m-1} \|_{L^\infty} \|u^m \|_{\dot B^{\frac{n+2}p,\frac{n+2}{2p} }_{p,1}  } \Big)\\
& \leq c \Big(  
\|   u^m \|_{\dot W^{1,\frac12}_{\frac{n+2}2}}  + 
\|   u^m \|_{\dot B^{1 + \frac{n+2}p, \frac12 + \frac{n+2}{2p}}_{p,1} } \Big)
\Big(  
\|   U^{m-1} \|_{\dot W^{1,\frac12}_{\frac{n+2}2}}  + 
\|   U^{m-1} \|_{\dot B^{1 + \frac{n+2}p, \frac12 + \frac{n+2}{2p}}_{p,1} } \Big)\\
& \leq c  A_m  B_{m-1},
\end{split}
\end{align}
where 
\begin{align*}
A_m&  = \| u^{m} \|_{\dot W^{1,  \frac12}_{\frac{n+2}2} } +  \| u^{m} \|_{\dot B^{1 +  \frac{n+2}p,  \frac12 + \frac{n+2}{2p}}_{p,1} } + \| u^{m} \|_{\dot B^{\al, \frac{\al}{2}}_{p,q} }, \\
B_m&  = \| U^{m} \|_{\dot W^{1,  \frac12}_{\frac{n+2}2} } +  \| U^{m} \|_{\dot B^{1 +  \frac{n+2}p,  \frac12 + \frac{n+2}{2p}}_{p,1} } + \| U^{m} \|_{\dot B^{\al, \frac{\al}{2}}_{p,q} }.
\end{align*}

\begin{align}
\begin{split}
  \|  u^{m} \otimes U^{m-1} \|_{  \dot B^{\al-1,\frac{\al-1}2}_{p,q}}  
  &\leq c \Big(\| u^m \|_{L^\infty}\| U^{m-1} \|_{\dot B^{\al-1, \frac{\al-1}{2}}_{p,q} }  + \| U^{m-1} \|_{L^\infty}\| u^m \|_{\dot B^{\al-1, \frac{\al-1}{2}}_{p,q} } \Big)\\
&\leq  c A_m  B_{m-1}.
\end{split}
\end{align}
\begin{align} \label{eq1109-3-3}
\begin{split}
 \| \si (Du^m) D U^{m-1} \|_{\dot B^{\frac{n+2}p,\frac{n+2}{2p} }_{p,1} } & \leq  c \Big(  \| \si (D u^m)\|_{L^\infty}  \| D U^{m-1} \|_{\dot B^{\frac{n+2}p,\frac{n+2}{2p}}_{p,1} } +  \| \si (D u^m)\|_{ \dot B^{\frac{n+2}p,\frac{n+2}{2p}}_{p,1} }  \| D U^{m-1} \|_{ L^\infty} \Big)\\
 & \leq  c \|   u^m\|_{ \dot B^{ 1 + \frac{n+2}p, \frac12 + \frac{n+2}{2p}}_{p,1} } \|  U^{m-1} \|_{\dot B^{1 + \frac{n+2}p, \frac12 + \frac{n+2}{2p}}_{p,1} }\\
 & \leq c  A_m  B_{m-1}.
 \end{split}
\end{align}
\begin{align}\label{0624-2-2}
\begin{split}
 \| \si (Du^m) D U^{m-1} \|_{\dot B^{\al -1,\frac{\al -1}2 }_{p,q} } & \leq  c  \| \si (Du^{m-1})  \|_{L^\infty } \| D U^{m-1} \|_{\dot B^{\al -1,\frac{\al -1}2 }_{p,q} } + \| \si (Du^m)  \|_{\dot B^{\al -1,\frac{\al -1}2 }_{p,q}  } \| D U^{m-1} \|_{ L^\infty }\\
& \leq  c  A_m  B_{m-1}.
\end{split}
\end{align}
Finally, 
\begin{align}
\begin{split}
& \| \big( \si (Du^m) - \si (Du^{m-1}\big) D u^m  \|_{\dot B^{\al -1,\frac{\al -1}2 }_{p,q} }\\
&   \leq  c  \| \big( \si (Du^m) - \si (Du^{m-1}\big)  \|_{L^\infty } \| D u^m \|_{\dot B^{\al -1,\frac{\al -1}2 }_{p,q} } + \|\big( \si (Du^m) - \si (Du^{m-1}\big)  \|_{\dot B^{\al -1,\frac{\al -1}2 }_{p,q}  } \| D u^m \|_{ L^\infty }.
\end{split}
\end{align}
The first term is dominated by 
\begin{align}
\|  DU^{m-1}  \|_{L^\infty } \| u^m \|_{\dot B^{\al,\frac{\al}2 }_{p,q} }  \leq  c \|  U^{m-1}  \|_{\dot B_{p,q}^{1 +\frac{n+2}p, \frac12 +\frac{n+2}{2p} }} \| u^m \|_{\dot B^{\al,\frac{\al}2 }_{p,q} } \leq  cA_m  B_{m-1}.
\end{align}
The second term is dominated by 
\begin{align}
\begin{split}
&  \Big( \| D U^{m-1}\|_{\dot B^{\al-1, \frac{\al-1}2}_{p,q}  } + \big( \|  D u^m\|_{\dot B^{\al -1, \frac{\al-1}2}_{p,q}  } +  \| D u^{m-1}\|_{\dot B^{\al -1, \frac{\al-1}2}_{p,q} }    \big) \big(  \| D U^{m-1}\|_{L^\infty  }  \big) \Big) \| u^m \|_{\dot B_{p,q}^{1 +\frac{n+2}p, \frac12 +\frac{n+2}{2p} }}\\
&  \leq  c A_m \Big( 1+ \big(A_m + A_{m-1}    \big)   \Big)   B_{m-1}.
\end{split}
\end{align}
Hence, we have 
\begin{align}
& \| \big( \si (Du^m) - \si (Du^{m-1}\big) D u^m  \|_{\dot B^{\al -1,\frac{\al -1}2 }_{p,q} }  \leq    c A_m \Big( 1+ \big(A_m + A_{m-1}    \big)   \Big)   B_{m-1}.
\end{align}
Similarly, we have 
\begin{align}
 \| \big( \si (Du^m) - \si (Du^{m-1}\big) D u^m  \|_{\dot B^{ \frac{n+2}p,\frac{n+2}{2p} }_{p,1} }
   \leq  c    A_m \Big( 1+ \big(A_m + A_{m-1}    \big)   \Big)   B_{m-1}.
\end{align}

Summing all estimates,  from Theorem \ref{thm-stokes},  we have
\begin{align}\label{0625-1}
\begin{split}
\| U^m\|_{\dot W^{1 ,\frac12}_{\frac{n+2}2}  } & \leq c \big(    \|u^m \otimes u^m - u^{m-1} \otimes u^{m-1}\|_{ L^{\frac{n+2}2} }  +   \| \si (Du^m) D u^m - \si (Du^{m-1}) D u^{m-1}\|_{ L^{\frac{n+2}2} } \big)\\
& \leq c \Big( \| u^{m}\|_{ \dot W^{1, \frac12}_{\frac{n+2}2}   } + \| u^{m-1}\|_{ \dot W^{1, \frac12}_{\frac{n+2}2}   } + \| D u^m\|_{L^\infty} + \| D u^{m-1}\|_{L^\infty}   \Big)  \| U^{m-1}\|_{\dot W^{1, \frac12}_{\frac{n+2}2} }\\
& \leq c\Big( A_m + A_{m-1} \Big)  B_{m-1},
\end{split}
\end{align}
\begin{align}
\begin{split}
 \|  U^{m} \|_{  \dot B^{1 + \frac{n+2}p, \frac12 + \frac{n+2}{2p}}_{p,1}} & \leq c \big(          \|u^m \otimes u^m - u^{m-1} \otimes u^{m-1}\|_{  \dot B^{\frac{n+2}p, \frac{n+2}{2p} }_{p,1}}  +   \| \si (Du^m) D u^m - \si (Du^{m-1}) D u^{m-1}\|_{ \dot B^{\frac{n+2}p,\frac{n+2}{2p} }_{p,1}} \big)\\
 &  \leq c\Big(  A_m + A_{m-1} \Big) B_{m-1},
 \end{split}
\end{align}
\begin{align}
\begin{split}
 \|  U^{m} \|_{  \dot B^{\al,\frac{\al}2}_{p,q}} & \leq c \Big(  \|u^m \otimes u^m- u^{m-1} \otimes u^{m-1}\|_{  \dot B^{\al -1,\frac{\al-1}2 }_{p,q} }  +   \| \si (Du^m) D u^m - \si (Du^{m-1}) D u^{m-1}\|_{ \dot B^{\al -1,\frac{\al-1}2 }_{p,q}} \Big)\\
 & \leq c\Big(  A_m + A_{m-1} + A_m (A_m + A_{m-1})  \Big) B_{m-1}.
 \end{split}
\end{align}
Taking $c\de_0 <  \frac16 $, we obtain
\begin{align}
B_m  < \frac12 B_{m-1}.
\end{align}
This implies that $\{ u_m\}$ is Cauchy sequence in $\ \dot W^{1,  \frac12}_{\frac{n+2}2} \cap  \dot B^{1 +  \frac{n+2}p,  \frac12 + \frac{n+2}{2p}}_{p,1} \cap \dot B^{\al, \frac{\al}{2}}_{p,q} $. In particular, $\{ u_m\}$ is Cauchy sequence in $ \dot W^{1,  \frac12}_{\frac{n+2}2}   \cap \dot B^{\al, \frac{\al}{2}}_{p,q} $. Since $\dot W^{1,  \frac12}_{\frac{n+2}2}   \cap \dot B^{\al, \frac{\al}{2}}_{p,q} $ is complete (see Theorem \ref{theo0715-1}), there is $u \in  \dot W^{1,  \frac12}_{\frac{n+2}2}   \cap \dot B^{\al, \frac{\al}{2}}_{p,q} $ such that $u_m$ converges to $u $ in  $  \dot W^{1,  \frac12}_{\frac{n+2}2}   \cap \dot B^{\al, \frac{\al}{2}}_{p,q} $.

\subsection{Existence}

Let $u$ be the same one constructed in  the previous section. Since $u^m \ri u$ in $\dot W^{1 ,\frac12}_{\frac{n+2}2}  \cap \dot B^{\al, \frac{\al}{2}}_{p,q} $, and by \eqref{230524-1}, we have that $u|_{x_n =0} = g$ and 
\begin{align*}
 \| u \|_{\dot W^{1 ,\frac12}_{\frac{n+2}2}  }, \quad \|u\|_{ \dot B^{\al,\frac{\al}2}_{p} }  <\de_0.
\end{align*}
Moreover, since $ D u^m \ri Du$ in $L^{\frac{n+2}2}$ and $ \| D u^m\|_{L^\infty} \leq  c \| u^m \|_{\dot B^{1 +\frac{n+2}p, \frac12 +\frac{n+2}{2p}}_{p,1}} \leq c\de_0$ for all $m$   by \eqref{230524-1},  we have $ \| D u \|_{L^\infty} \leq c \de_0$.

In this section, we will show that $u$ satisfies weak formulation  \eqref{weaksolution-NS}, that is, $u$ is a weak solution of  \eqref{maineq2} with appropriate distribution $p$.
Let $\Phi\in C^\infty_{0}({\R_+} \times [0,\infty))$ with $\mbox{div }\Phi=0$.
Observe that
\begin{align*}
\int_{\R_+}u_0(x) \cdot \Phi(x) dx &=\int^\infty_0\int_{\R_+} -u^{m+1}\cdot \Phi_t+ \big( (u^m\otimes u^m): \nabla \Phi + S(D u^m)   \big): \nabla \Phi dxdt.
\end{align*}
Note that $u^m \ri u$ in $\dot W^{1 ,\frac12}_{\frac{n+2}2}  $ and so $u^m \ri u$ in  $ L^{n+2}  $. 

Let $ K $ be a compact subset of  $\R_+ \times (0, \infty)$. Note that  
\begin{align*}
 \| u^m \otimes u^m - u \otimes u \|_{L^1  (K)}
  &   \quad\leq
\| u^m\otimes ( u^m -u)  +( u^m - u )\otimes u \|_{L^1  (K)}\\
  & \quad \leq
c  \big( \| u^m\|_{L^2   (K)   }  + \| u\|_{L^2  (K)   }\big)  \|( u^m -u) \|_{L^2  (K)}\\
  & \quad \leq
c(K)  \big( \| u^m\|_{L^{n+2}     } +  \| u\|_{L^{n+2}    }  \big)   \|( u^m -u) \|_{L^{n+2}  }\\
  & \quad \leq
c(K) \de_0  \| u^m -u\|_{L^{n+2} },\\
\|  \si (D u^m) D u^m -  \si (D u) D u \|_{L^1  (K)}
 &  \quad\leq
c \big(\| Du \|_{L^\infty} + \| D u^m\|_{L^\infty}   \big) \|( Du^m -Du) \|_{L^1  (K)}\\
&  \quad\leq
c(K) \de_0 \|( Du^m -Du) \|_{L^{\frac{n+2}2}  }.
\end{align*}
This implies   $u^m  \otimes u^m$ and  $  S (D u^m)   $ converge to $ u\otimes u$ and $S(D u)   $ in $L^1 (K)$ for all bounded subset $K \subset \R_+\times (0, \infty)$, respectively. With $m$  sending to   infinity,   the identity is
\begin{align*}
\int_{\R}u_0(x) \cdot \Phi(x) dx &=\int^\infty_0\int_{\R_+} - u \cdot \Phi_t+ \big( (u\otimes u): \nabla \Phi + S(D u) \big) : \nabla \Phi dxdt.
\end{align*}
Therefore  $u$ is a weak solution of \eqref{maineq2}.  This completes the proof of the existence part of Theorem \ref{thm-navier}.

\subsection{Uniqueness }
Let $ u_1\in   \dot W^{1, \frac12}_{\frac{n+2}2} \, \cap  \,  u \in L^\infty(0, \infty, \dot W^1_\infty(\R_+)  )  $ be  another weak solution   of the system \eqref{maineq2} with pressure $p_1$ satisfying $\| u_1 \|_{\dot B^{1 + \frac{n+2}p, \frac12 + \frac{n+2}{2p}}_{p,1} } + \| D u_1\|_{L^\infty}  < \de_0$, where $\ep$ is chosen in Section \ref{0714-1}. Let $U = u - u_1$ and $P = p -p_1$. According,
 $(U, P)$ satisfies the equations
\begin{align*}
U_t - \De U + \na P& =-\mbox{div}(u\otimes u-u_1\otimes u_1)  -\mbox{div}(  \si(D u) Du  -\si( D  u_1) D u_1),\\
 {\rm div} \, U& =0,
 \mbox{ in }\R_+\times (0,\infty),\\
 U|_{t=0}= 0, &\quad U|_{x_n =0} =0.
\end{align*}

The estimate  \eqref{0625-1}  implies that
\begin{align*}
  \|  U \|_{ \dot W^{1, \frac12}_{\frac{n+2}2} }
   &  \leq c  \Big( \| u\|_{ \dot W^{1, \frac12}_{\frac{n+2}2}   } + \| u_1\|_{ \dot W^{1, \frac12}_{\frac{n+2}2}   } + \| D u\|_{L^\infty}  + \| D u_1\|_{L^\infty}   \Big) \| U\|_{ \dot W^{1, \frac12}_{\frac{n+2}2}  }\\
&  < c \de_0\| U \|_{ \dot W^{1, \frac12}_{\frac{n+2}2} }  <\frac12 \| U \|_{ \dot W^{1, \frac12}_{\frac{n+2}2} }.
 \end{align*}
This implies that $u \equiv u_1$ in $\R_+ \times (0, \infty)$. Thus, the proof of the uniqueness  of Theorem \ref{thm-navier} is successfully proved.

\appendix

\section{ Proof of (2) of Remark \ref{rem0712}}
\label{appendix240721}
\setcounter{equation}{0}

First, we show that ${\mathcal S}_0(\RN) = \{ f \in {\mathcal S}(\RN) \, | \, 0 \notin supp \, \hat f\}$ is dense in $W^{s, \frac{s}2}_p (\RN)$. If then, $  \sum_{-\infty< k < \infty} (  4\pi^2 |\xi| +2\pi i \tau)^\frac{s}2  \hat \phi(2^{-k} \xi, 2^{-2k}\tau) \hat{f}  =   (  4\pi^2 |\xi| +2\pi i \tau)^\frac{s}2  \hat{f}$ and so (2) of Remark \ref{rem0712} is proved.

To prove the claim, let us $ f \in \dot W^{s, \frac{s}2}_p (\RN)$ and $\ep > 0$. Let $ g \in L^p (\RN)$ be defined by $ {\mathcal F}(g)(\xi, \tau) = {\mathcal F}((4\pi^2 |\xi |^2 + 2\pi i \tau)^\frac{s}2 \hat f)(\xi, \tau)$.   Since ${\mathcal S}_0(\RN) $ is dense in $L^p (\RN)$ (see Lemma 3.6 in \cite{DHMT}), there is $ g_\ep \in {\mathcal S}_0(\RN) $ such that $\| g -g_\ep\|_{L^p (\RN)}  < \ep$.  Since $0 \notin supp\,\, \hat g_\ep $,   we get $ f_\ep: ={\mathcal F}^{-1} (4\pi^2 |\xi |^2 + 2\pi i \tau)^{-\frac{s}2} \hat g_\ep) \in {\mathcal S}_0 (\RN)$. Then, we have $ \| f_\ep - f\|_{\dot W^{s,\frac{s}2}_p (\RN)} = \| g - g_\ep \|_{L^p (\RN)} < \ep$. Hence, ${\mathcal S}_0 (\RN)$ is dense in $ W^{s, \frac{s}2}_p (\RN)$. We complete the proof of (2) of Remark \ref{rem0712}.

\section{}
%{Proof of \eqref{0526-1} }
\label{appendix0-0-0}
\setcounter{equation}{0}

In this section, we prove the following theorem. 

\begin{theo}\label{proofinR}
\begin{itemize}
\item[(1)]
For  $ 1 \leq p < \infty$ and $ k \in {\mathbb N}$,  let us
\begin{align}\label{0712-3}
\dot H^{k, \frac{k}2}_p (\RN ) = \{ f \in {\mathcal S}'  (\RN  )\, | \,  \| f\|_{\dot H^{k, \frac{k}2}_p (\RN )}  :  =\sum_{|\be| +  2l =k} \|D_x^\be D_t^l f\|_{L^p (\RN )} < \infty \}.
\end{align}
Then, $ \dot W^{k, \frac{k}2}_p (\RN ) = \dot H^{k, \frac{k}2}_p (\RN )$ with  $\|f\|_{ \dot W^{k, \frac{k}2}_p (\RN ) } =  \|f\|_{ \dot H^{k, \frac{k}2}_p (\RN ) }$.  Moreover, 
\begin{align}\label{0712-2}
\|f \|_{\dot W^{2k, k }_p (\RN ) } \approx   \sum_{k_1 + k_2 =k} \|\De^{k_1} D_t^{k_2} f\|_{L^p (\RN )}.
\end{align}
 
 \item[(2)]
Let $ 1 \leq p < n+2  $  and $0< s$ with $ p < \frac{n+2}s$.  Then,  $  \dot W^{s, \frac{s}2}_p(\RN ) \subset L^{\frac{p(n+2)}{n +2-ps}} (\RN ) $ with 
\begin{align*}
\| f\|_{L^{\frac{p(n+2)}{n +2 -ps}} (\RN )  } \leq c \| f \|_{\dot W^{s,\frac{s}2}_p (\RN )}.
\end{align*}

 \item[(3)]
 Let $ 1 \leq p < n +2 $, $ 1 \leq q  < \infty $ and $0< s$ with $ p < \frac{n+2 }s$. Then,  $\dot W^{s,\frac{s}2}_p (\RN )$ and $\dot B^{s,\frac{s}2}_{p,q} (\R \times {\mathbb R})$ are completion of $C^\infty_0(\RN )$. In particular,  $\dot W^{s,\frac{s}2}_p (\RN )$ and $\dot B^{s,\frac{s}2}_{p,q} (\R \times {\mathbb R})$ are complete.

\end{itemize}
\end{theo}

\begin{proof}
H. Bahouri, J. Y. Chemin and R. Danchin proved Theorem \ref{proofinR} for $ \dot H^{s }_2 (\R  ), \,\, s < \frac{n}2 $ (see 
   the comment  in  Proposition 1.32 below,  Theorem 1.38 and Proposition 1.34  in   \cite{danchin} ).

Since the proofs are similar, we only  prove (1) for $k =1$.   Let $ f \in {\mathcal S}' (\RN)$ such that $\sum _{ 1 \leq i \leq n} \| D_{x_i}  f\|_{L^p (\RN)} + \| D_t^\frac12 f\|_{L^p (\RN)} <\infty$. 
Then,  we have   
\begin{align*}
\begin{split}
 (  4\pi^2   |\xi|^2 +2\pi i \tau)^\frac12 \widehat{ f} 
 &  = -\sum_{1 \leq i \leq n}  \frac{ 2\pi i \xi_i}{ ( 4\pi^2  |\xi|^2 +2\pi i \tau)^\frac12} \widehat{ D_{x_i}   f }\\
 & \qquad +  2\pi i\frac{ \big(a_0 + i b_0 sign(\tau) \big)^{-1}sign(\tau) |\tau|^{-\frac12}  }{ (  4\pi^2 |\xi|^2 +2\pi i \tau)^\frac12}  \widehat{ D_t^\frac12   f }.
 \end{split}
\end{align*}
Note that by the Marcinkiewicz multiplier theorem (see Theorm 4. 6$'$  in \cite{St}), $\frac{\xi_i}{ (   |\xi|^2 +2\pi i \tau)^\frac12}$ and $\frac{\big(a_0 + i b_0 sign(\tau) \big)^{-1}sign(\tau) |\tau|^{-\frac12}  }{ (   |\xi|^2 +2\pi i \tau)^\frac12}$ are $L^p-$multipliers. Hence, we have 
\begin{align*}
\begin{split}
\| f\|_{\dot W^{1,\frac12}_p (\RN)} &  = \| {\mathcal F}^{-1} ((  4\pi^2   |\xi|^2 +2\pi i \tau)^\frac12  \hat f)\|_{L^p (\RN)}\\
& \leq c \big(\sum _{ 1 \leq i \leq n} \| D_{x_i}  f\|_{L^p (\RN )} + \| D_t^\frac12 f\|_{L^p (\RN)} \big).
\end{split}
\end{align*}
Hence, we have  $ \dot H^{1,\frac12}_p (\RN) \subset \dot W^{1,\frac12}_p (\RN)  $. 

Conversely, let $f \in \dot W^{1,\frac12}_p (\RN)$ so that $ f \in {\mathcal S}' (\RN)$ and   $\| {\mathcal F}^{-1} ((  4\pi^2   |\xi|^2 +2\pi i \tau)^\frac12  \hat f)\|_{L^p (\RN)} < \infty$.  Then, we have 
\begin{align*}
\widehat{ D_{x_i}  f }  = (2\pi i \xi_i ) \widehat{f} = - \frac{2\pi i \xi_i}{( 4\pi^2   |\xi|^2 +2\pi i \tau)^\frac12}(  4\pi^2 |\xi|^2 +2\pi i \tau)^\frac12 \hat{f},\\
\widehat{ D_t^\frac12  f }  = (2\pi i |\tau|^\frac12 ) \hat{f} =  \frac{\big(a_0 + i b_0 sign(\tau) \big) |\tau|^\frac12 }{(  4\pi^2 |\xi|^2 +2\pi i \tau)^\frac12}( 4\pi^2  |\xi|^2 +2\pi i \tau)^\frac12 \hat{f}.
\end{align*}
By the Marcinkisch multiplier theorem (see Theorm 4. 6$'$  in \cite{St}), $ \frac{\big(a_0 + i b_0 sign(\tau) \big) |\tau|^\frac12 }{(  4\pi^2 |\xi|^2 +2\pi i \tau)^\frac12}$ is $L^p-$ multiplier. Hence,  we have 
\begin{align*}
\| D_{x_i}  f  \|_{L^p (\RN)}, \,\, \| D_t^\frac12 f \|_{L^p (\RN)} \leq c \| f\|_{\dot W^{1,\frac12}_p (\RN)}.
\end{align*}
Hence, we have  $ \dot W^{1,\frac12}_p (\RN) \subset \dot H^{1,\frac12}_p (\RN) $.  We complete the proof of \eqref{0712-3}.

Next, 
\begin{align*}
\widehat{ D_{x_i}D_{x_j}  f }  = (2\pi i \xi_i ) \widehat{f} = - \frac{-4\pi^2 \xi_i \xi_j}{( 4\pi^2   |\xi|^2 +2\pi i \tau) }(  4\pi^2 |\xi|^2 +2\pi i \tau)  \hat{f}.
\end{align*}
This implies \eqref{0712-2}. We complete the proof of (1).

Now, we prove (2).  Let $ f \in \dot H^{2,1}_p (\RN)$ so that $ f_t -\De f \in L^p (\RN)$.  Note that  
$ f = \Ga* \big(f_t -\De f\big)  $, where $\Ga$ is the fundamental solution of heat equation in $\R$. It is well-known that 
\begin{align*}
\|\Ga* \big(f_t -\De f\big) \|_{L^{\frac{p (n+2)}{n+2 -2p}} (\RN)} \leq c \| f_t -\De f  \|_{L^p (\RN)}.
\end{align*}
Hence, we have \begin{align}\label{0712-1}
\| f\|_{L^{\frac{p (n+2)}{n+2 -2p}} (\RN)}\leq c\| f\|_{\dot H^{2,1}_p (\RN)}.
\end{align}

Let $ m \in {\mathbb N}$ such that $ 2m \leq n+2 < 2(m+1)$ and $ p < \frac{n+2}{2m}$. Then, we have 
\begin{align*}
\begin{split}
\| f\|_{L^{\frac{p (n+2)}{n+2 -2m p}} (\RN)}& \leq c\| f_t -\De f  \|_{L^{\frac{p (n+2)}{n+2 -2(m-1) p}} (\RN)}\\
&  \leq c\| D_t(D_t f -\De f) - \De(D_t f  -\De f)  \|_{L^{\frac{p (n+2)}{n+2 -2(m-2) p}} (\RN)}\\
&   \hspace{40mm} \cdots\\
&\leq c \| \sum_{k_1+  l = m} D_t^l \De^{k_1} f\|_{L^p (\RN)}\\
&\leq c \|  f\|_{\dot W^{2m, m}_p  (\RN)}.
\end{split}
\end{align*}

 Acting  complex  interpolation between $ L^{p_1} (\RN)$ and $ L^{\frac{(n+2)p}{n +2- 2mp}}(\RN)$ for $ 1 \leq p_1 < \infty$ and $1 \leq p < \frac{n +2}{2m}$,  we have 
 \begin{align*}
 \| f\|_{[L^{p_1} (\RN), L^{\frac{(n+2)p}{n +2 -2mp}}(\RN)]_\te} \leq c \| f \|_{[L^{p_1} (\RN), \dot W^{2m,m}_p (\RN)]_\te}.
  \end{align*} 
  Hence, we obtan (2) for $ 0  <  s \leq 2m$ and $1 \leq p \leq  \frac{n+2}{s}$.
  
 For the case $ 2m< s< n +2$, we use the followng calim:
$ f \in \dot W^{s,\frac{s}2}_p (\RN)$ if and only if $D^{2m}_x f, \,\, D_t^m f \in \dot W^{s -2m,\frac{s -2m}2}_p (\RN)$.

The claim is direct from the following equation.
\begin{align*}
(4\pi^2 |\xi|^2 + 2\pi  i \tau)^\frac{s -2m}2  \widehat{\De^m f}   = \frac{(-4\pi^2 |\xi|^2)^m}{(4\pi^2 |\xi|^2 + 2\pi i \tau)^m}  
(4\pi^2 |\xi|^2 + 2\pi i \tau)^\frac{s }2   \widehat{f},\\
(4\pi^2 |\xi|^2 + 2\pi  i \tau)^\frac{s -2m}2  \widehat{D_t^m f}   = \frac{(2\pi i \tau)^m}{(4\pi^2 |\xi|^2 + 2\pi i \tau)^m}  
(4\pi^2 |\xi|^2 + 2\pi i \tau)^\frac{s }2   \widehat{f}. 
\end{align*}

Let $ f \in \dot W^{s,\frac{s}2}_p (\RN)$ for  $  s <n +2$ and $1 \leq p < \frac{n+2}s$. From the claim, we have  the claim, we have  $\De f, \,\, D_t f  \in \dot W^{s -2,\frac{s -2}2}_p (\RN)$ with 
\begin{align*}
\| \De  f \|_{\dot W^{s-2,\frac{s -2}2}_p (\RN)}  + \| D_t  f \|_{\dot W^{s-2,\frac{s -2}2}_p (\RN)}\approx \|  f \|_{\dot W^{s,\frac{s}2}_p (\RN)}.
\end{align*}

From    \eqref{0712-1}, \eqref{0712-3},  (2) of Theorem \ref{proofinR} for $ 0 < s\leq 2m $ and the claim continuosly,     we have 
\begin{align*}
\begin{split}
\|  f \|_{L^{\frac{(n+2)p}{n+2 -sp} (\RN)}} &  \leq c\|  f \|_{\dot W^{2,1}_{\frac{p (n+2)}{n +2 -(s -2)p}} (\RN)} \leq c \| \De  f \|_{L^{\frac{p(n+2)}{n +2 -(s-2)p}} (\RN)}  + \| D_t  f \|_{L^{\frac{p(n+2)}{n +2 -(s-2)p}}(\RN)}\\
&  \leq c \| \De  f \|_{\dot W^{s-2,\frac{s -2}2}_p (\RN)}  + \| D_t  f \|_{\dot W^{s-2,\frac{s -2}2}_p (\RN)}\approx \|  f \|_{\dot W^{s,\frac{s}2}_p (\RN)}.
\end{split}
\end{align*}
This implies the (2) of Theorem \ref{proofinR} for $\dot H^{s,\frac{s}2}_p (\RN)$. 

For the proof of  (2) of Theorem \ref{proofinR} for $\dot B^{s,\frac{s}2}_p (\RN)$, we use the following real interpolation  
\begin{align}\label{240703-2}
( L^p (\R), \dot W^{s_1,\frac{s_1}2}_p (\RN)_{\te,q} = \dot B^{s,\frac{s}2}_{p,q} (\RN),
\end{align}
where  $ \te =1 -\frac{s}{s_1}$. 
We complete the proof of  (2) of Theorem \ref{proofinR}.

The completeness of $\dot W^{s,\frac{s}2}_p(\RN)$  is direct result from (2) of Theorem \ref{proofinR}.  
The completeness of $\dot B^{s,\frac{s}2}_{p,q} (\RN)$ is from \eqref{240703-2} and the property of interpolation spaces. We complete the proof of (3) of Thererem \ref{proofinR}. 
\end{proof}

\begin{theo}\label{theo0715-1}
Let  $1\leq  p< n+2$,  $ \frac{n+2}{n+\al} <  r < \infty$, $ 1  \leq q < \infty$ and  $  1 < \al <2 $. Then, $ \dot W^{1,\frac{1}2}_p(\RN) \cap \dot B^{\al,\frac{\al}2}_{r,q}(\RN)$ are complete.
\end{theo}
\begin{proof}
Note that     $f \in \dot B^{\al, \frac{\al}2}_{r,q}(\RN)$  if and only if $ {\mathcal F}^{-1}( (4\pi^2 |\xi|^2 +2\pi i \tau)  \widehat{f} ) \in \dot B^{\al -2, \frac{\al-2}2}_{r,q} (\RN)$ with $\| f\|_{\dot B^{\al, \frac{\al}2}_{r,q}(\RN)} \approx  \| {\mathcal F}^{-1}( (4\pi^2 |\xi|^2 +2\pi i \tau)  \widehat{f} )\|_{\dot B^{\al -2, \frac{\al-2}2}_{r,q} (\RN) }$.

Since $ \dot B^{\al -2, \frac{\al-2}2}_{r,q} (\RN)$ is dual space of $\dot B^{-\al +2, \frac{-\al+2}2}_{r',q'} (\RN)$ and  $\dot B^{-\al +2, \frac{-\al+2}2}_{r',q'} (\RN)$ is complete  by  Theorem \ref{proofinR}, $ \dot B^{\al -2, \frac{\al-2}2}_{r,q} (\RN)$ is complete.

Let $ \{ f_m \}$ be a Cauchy sequence in $ \dot W^{1,\frac{1}2}_p(\RN) \cap \dot B^{\al,\frac{\al}2}_{r,q}(\RN)$. Then,  $ \{ f_m \}$ be a Cauchy sequence in $ \dot W^{1,\frac{1}2}_p(\RN) $ and  $ \{{\mathcal F}^{-1}( (4\pi^2 |\xi|^2 +2\pi i \tau)  \widehat{f_m} ) \}$ be a Cauchy sequence in $  \dot B^{\al -2,\frac{\al -2}2}_{r,q}(\RN)$. Since  $ \dot W^{1,\frac{1}2}_p(\RN) $ and  $  \dot B^{\al -2,\frac{\al -2}2}_{r,q}(\RN)$ are complete, there are $f \in \dot W^{1,\frac{1}2}_p(\RN) $ and $ g \in \dot B^{\al -2,\frac{\al -2}2}_{r,q}(\RN)$ such that $ f_m$ converges to $f$ in  $f \in \dot W^{1,\frac{1}2}_p(\RN) $ and $ \{{\mathcal F}^{-1}( (4\pi^2 |\xi|^2 +2\pi i \tau)  \widehat{f_m} ) \}$ converges to $g$ in $ \dot B^{\al -2,\frac{\al -2}2}_{r,q}(\RN)$.

Let $\phi \in  {\mathcal S} (\RN)$. 
Note that $  {\mathcal F}^{-1}( (4\pi^2 |\xi|^2 +2\pi i \tau)  \widehat{\phi} \in {\mathcal S} (\RN)$ and  $< {\mathcal F}^{-1}( (4\pi^2 |\xi|^2 +2\pi i \tau)  \widehat{f_m} ), \phi> = <f_m,  {\mathcal F}^{-1}( (4\pi^2 |\xi|^2 +2\pi i \tau)  \widehat{\phi} )> $, where $<\cdot, \cdot>$ is dual pairing between ${\mathcal S}(\RN)$ and ${\mathcal S}(\RN)'$.
 Then, we have 
\begin{align}
\begin{split}
<  {\mathcal F}^{-1}( (4\pi^2 |\xi|^2 +2\pi i \tau)  \widehat{f_m} ),    \phi >   & \ri < g,  \phi>,\\
< f_m,    {\mathcal F}^{-1}( (4\pi^2 |\xi|^2 +2\pi i \tau)  \widehat{\phi} )>  & \ri <f,  {\mathcal F}^{-1}( (4\pi^2 |\xi|^2 +2\pi i \tau)  \widehat{\phi} )>\\
 &= <   {\mathcal F}^{-1}( (4\pi^2 |\xi|^2 +2\pi i \tau)\widehat{f} ), \phi>.
 \end{split}
\end{align}
This implies  $ g =    {\mathcal F}^{-1}( (4\pi^2 |\xi|^2 +2\pi i \tau)  \widehat{f} )$ and so  $ f =    {\mathcal F}^{-1}( (4\pi^2 |\xi|^2 +2\pi i \tau)^{-1} \widehat{g} )$. Hence, we have $ f \in  \dot W^{1,\frac{1}2}_p(\RN) \cap \dot B^{\al, \frac{\al}2}_{r,q} (\RN)$.   We comlete the proof of Theorem \ref{theo0715-1}.

\end{proof}

\section{Proof of Proposition \ref{proppoisson2}  }
\label{proofoflemma230518-0}
\setcounter{equation}{0} 
Note that $D_{x_i} N*' f = D_{x_n} N*' R'_i f, \,\, i = 1, \cdots, n-1$ and $D_{x_n} N  $ is Poisson kernel in $\R_+$, where $R'_i$ are $(n-1)$-diminsional Riesz transforms. Moreover, it is  a well known fact that  $D_{x_n} N$ is bounded from $\dot{B}^{-\frac{1}{p}}_{p, p}(\Rn)$ to $L^p(\R_+)$ so that $
  \|D_{x_n} *' f\|_{L^p(\R_+)}\leq c\|f\|_{\dot B^{-\frac{1}{p}}_{p, p}(\Rn)}$
 and $R'_i$ are bounded in $\dot B^{\be}_{p,p}(\Rn)$ to $\dot B^{\be}_{p,p} (\Rn), \,\, \be \in {\mathbb R}$  (See \cite{St}).
Accordingly, we have   
\begin{align*}
\|  N*' f\|_{L^p ( 0, \infty; \dot W^k_p (\R_+))} \leq c \|R'_i f\|_{L^p(0, \infty;\dot B^{k-\frac{1}{p}}_{p,p}(\Rn))} \leq c \|f\|_{L^p(0, \infty;\dot B^{k-\frac{1}{p}}_{p,p}(\Rn))}.
\end{align*}
Since $D_t^k N*' f = N*' D_t^k f$, we have 
\begin{align*}
\|  N*' f\|_{L^p (\R_+; \dot W^{k}_p (0, \infty))} \leq c \|  N*'  D^k_t f\|_{L^p(\R_+ \times (0, \infty))  } \leq c \| D^k_t f\|_{ L^p(0, \infty;\dot B^{-\frac{1}{p}}_{p,p}(\Rn))} \leq c \|f\|_{\dot W^{k}_p(0, \infty;\dot B^{-\frac{1}{p}}_{p,p}(\Rn))}.
\end{align*}
From (3) of Proposition \ref{prop0215}, this implies the first quantity of Proposition \ref{proppoisson2}.

The second quantity is from the property of real interpolation and the definition of space $\dot A^{\al, \frac{\al}2 }_{p,q} $. We complete the proof of  Proposition \ref{proppoisson2}.

\section{Proof of Proposition \ref{proppoisson0518}  }
\label{proofoflemma230518-1}
\setcounter{equation}{0}

Since the proofs are exactly same, we only prove in the case of $ \na^2  N* g$.  By the Caldr\'{o}n-Zygmund integral theorem and by the property of  real interpolation, 
\begin{align}\label{230528-1}
\|\na^2  N* g\|_{ L^p (0, \infty; \dot W^k_p  (\R_+))}&\leq c \|g\|_{ L^p (0, \infty;\dot W^{k} _p (\R_+))} \quad   k  \geq 0.
%\|\na^2 N* g\|_{L^p (0, \infty; \dot B^{s}_{p,q}(\R_+))} &\leq c   \|g\|_{L^p (0, \infty;\dot B^{s}_{p,q}(\R_+))} \quad    s > 0.
\end{align}
See Proposition 3.4 in \cite{CJ3}.

Let $\tilde g \in \dot W^{\frac{k}2}_p ({\mathbb R})$ be a extension of $g \in \dot W^{\frac{k}2}_p (0, \infty)$ such that $\| \tilde g\|_{\dot W^{\frac{k}2}_p ( {\mathbb R}) } \leq c \| g\|_{\dot W^{\frac{k}2}_p (0, \infty) } $. Accordingly
\begin{align}\label{230528-2}
\begin{split}
\|\na^2  N* g\|_{ L^p (\R_+;  \dot W^{\frac{k}2}_p  (0, \infty))} \leq c \|\na^2  N*\tilde  g\|_{ L^p (\R_+;  \dot W^{\frac{k}2}_p  ({\mathbb R}))}
= c \|\na^2  N*  D_t^{\frac{k}2} \tilde g\|_{ L^p (\R_+\times {\mathbb R})}\\
 \leq  c \| \tilde g\|_{L^p (\R_+;  \dot W^{\frac{k}2}_p  ({\mathbb R}))} \leq c \| g\|_{L^p(\R_+;  \dot W^{\frac{k}2}_p  (0,\infty))} .
 \end{split}
\end{align}
From \eqref{230528-1},  \eqref{230528-2} and  (3) of Proposition \ref{prop0215},  
\begin{align*}
\|\na^2  N* g\|_{ \dot W^{k,\frac{k}2} _p  } &\leq c \|g\|_{ \dot W^{k,\frac{k}2} _p  } \quad   k  \geq 0.
\end{align*}
Using the property of real interpolation ,  
\begin{align*}
\|\na^2 N* g\|_{ \dot B^{s, \frac{s}2}_{p,q} } &\leq c   \|g\|_{\dot B^{s, \frac{s}2}_{p,q} } \quad    s > 0.
\end{align*}
We complete the proof of Proposition \ref{proppoisson0518}.

\section{Proof of Lemma \ref{lemma0716}  }
\label{proofoflemma0716}
\setcounter{equation}{0}

Let $E : \dot B^{s,\frac{s}2}_{p,q} \ri \dot B^{s,\frac{s}2}_{p,q} (\RN)$ be an extension operator defined in Section  \ref{notation}.
Let  $ {\mathbb G} \in \dot B^{s,\frac{s}2}_{p,q} \cap L^\infty$ and $ \tilde{\mathbb G} = E {\mathbb G}$ such that  $ \| \tilde{\mathbb G} \|_{ \dot B^{s,\frac{s}2}_{p,q}(\RN) } \leq c \| {\mathbb G}\|_{ \dot B^{s,\frac{s}2}_{p,q}}$ and $ \| \tilde{\mathbb G} \|_{ L^\infty (\RN) } \leq c \| {\mathbb G}\|_{L^\infty}$.  Since $\si (\tilde {\mathbb G}) \tilde {\mathbb G}|_{\R_+ \times (0, \infty)} = \si (  {\mathbb G})   {\mathbb G}$, by he definition of  definition of homogeneous anisotropic Besov space  $\dot B^{s,\frac{s}2}_{p,q}$ and  from  (4) of Proposition \ref{prop2},   we have 
\begin{align*}
\| \si ({\mathbb G}){\mathbb G}\|_{\dot B^{s,\frac{s}2}_{p,q} } & \leq c\| \si (\tilde {\mathbb G}) \tilde {\mathbb G}\|_{\dot B^{s,\frac{s}2}_{p,q} (\RN) }\\
  & \leq c    \Big(   \int_{{\mathbb R}^{n +1}}  
\frac1{(|y| +|\tau|^\frac12)^{n+2  + ps  } } \\
& \qquad  \times  \big( \int_{{\mathbb R}^{n +1}}  |\si ( \tilde{\mathbb G}(x+y,t+\tau)) \tilde{\mathbb G}(x +y,t+\tau) - \si (\tilde{\mathbb G}(x,t)) \tilde{\mathbb G} (x,t)|^p  dxdt \big)^{\frac{q}p}  dyd\tau  \Big)^\frac1q.
\end{align*}

By mean-value theorem  $| \si (\tilde {\mathbb G}(x+y,t+\tau))| 
\leq \| D \si\|_{L^\infty(0, \| {\mathbb G}\|_{L^\infty(\R_+)})} | \leq  \tilde{\mathbb G}(x+y,t+\tau)|  $ and $| \si (\tilde {\mathbb G}(x+y,t+\tau))  - \si (\tilde{\mathbb G}(x,t))    | \leq c \| D \si\|_{L^\infty(0, \| {\mathbb G}\|_{L^\infty(\R_+)})}   |\tilde{\mathbb G}(x+y,t+\tau) - \tilde{\mathbb G}(x,t)|  $. Hence, we have 
\begin{align*}
&|\si (\tilde {\mathbb G}(x+y,t+\tau)) \tilde{\mathbb G}(x +y,t+\tau) - \si (\tilde{\mathbb G}(x+y,t +\tau)) \tilde{\mathbb G} (x,t)|\\
& \leq  |\si (\tilde {\mathbb G}(x+y,t+\tau)) \Big( \tilde{\mathbb G}(x +y,t+\tau) -   \tilde{\mathbb G} (x,t) \Big)  | + |\Big(\si (\tilde {\mathbb G}(x+y,t+\tau))  - \si (\tilde{\mathbb G}(x,t))  \Big)\tilde{\mathbb G} (x,t) |\\
& \leq c \| D \si\|_{L^\infty(0, \| {\mathbb G}\|_{L^\infty(\R_+)})}  \| {\mathbb G}\|_{L^\infty(\R_+)}    |\tilde {\mathbb G}(x+y,t +\tau) - \tilde{\mathbb G}(x,t)|.
\end{align*}

\begin{align*}
\| \si ({\mathbb G}){\mathbb G}\|_{\dot B^{s,\frac{s}2}_{p,q} }
& \leq c\| D \si\|_{L^\infty(0, \| {\mathbb G}\|_{L^\infty(\R_+)})}  \| {\mathbb G}\|_{L^\infty(\R_+)}\\
& \quad \times   \Big(   \int_{{\mathbb R}^{n+1}}  
\frac1{(|y| +|\tau|^\frac12)^{n+2  + ps  }}  \big( \int_{{\mathbb R}^{n+1}}  |  \tilde{\mathbb G}(x +y,t+\tau) -   \tilde{\mathbb G} (x,t)|^p  dxdt \big)^{\frac{q}p}  dyd\tau  \Big)^\frac1q\\
&\leq c\| D \si\|_{L^\infty(0, \| {\mathbb G}\|_{L^\infty(\R_+)})}  \| {\mathbb G}\|_{L^\infty(\R_+)} \| \tilde{\mathbb G}\|_{\dot B^{s, \frac{s}2}_{p,q} (\RN)}\\
&\leq c\| D \si\|_{L^\infty(0, \| {\mathbb G}\|_{L^\infty(\R_+)})}  \| {\mathbb G}\|_{L^\infty(\R_+)} \| {\mathbb G}\|_{\dot B^{s, \frac{s}2}_{p,q} }.
\end{align*}
We complete the proof of (1) of  Lemma \ref{lemma0716}.

Next, we prove (2).  By direct caculation, we have 
%\begin{align}
%\begin{split}
%&\si \big( \tilde{\mathbb G}(x+y,t+\tau)) - \si ( \tilde{\mathbb H}(x +y,t+\tau) \big) \\
%& =\int_0^1 \frac{d}{d \te} \si \big(\te \tilde{\mathbb G}(x+y,t+\tau) + (1 -\te) \tilde{\mathbb H}(x +y,t+\tau) \big) d\te\\
%& =\int_0^1 D\si \big(\te \tilde{\mathbb G}(x+y,t+\tau) + (1 -\te) \tilde{\mathbb H}(x +y,t+\tau) \big): \Big( {\mathbb G}(x+y, t+\tau) -  {\mathbb H}(x+y, t+\tau) \Big) d\te
%\end{split}
%\end{align}
%and
%\begin{align}
%\si ( \tilde{\mathbb G}(x,t)) - \si ( \tilde{\mathbb H}(x,t) )   =\int_0^1 D\si \big(\te \tilde{\mathbb G}(x,t) + (1 -\te) \tilde{\mathbb H}(x,t) \big): \Big( {\mathbb G}(x, t) -  {\mathbb H}(x, t) \Big) d\te.
%\end{align}
\begin{align*}
\begin{split}
&\si ( \tilde{\mathbb G}(x+y,t+\tau)) - \si ( \tilde{\mathbb H}(x +y,t+\tau) ) - \Big( \si ( \tilde{\mathbb G}(x,t)) - \si ( \tilde{\mathbb H}(x,t) )  \Big)\\
& = \int_0^1 D\si (\te \tilde{\mathbb G}(x+y,t+\tau) + (1 -\te) \tilde{\mathbb H}(x +y,t+\tau) ): \Big( {\mathbb G}(x+y, t+\tau) -  {\mathbb H}(x+y, t+\tau) \Big) d\te\\
& - \int_0^1 D\si (\te \tilde{\mathbb G}(x,t) + (1 -\te) \tilde{\mathbb H}(x,t) ): \Big( {\mathbb G}(x, t) -  {\mathbb H}(x, t) \Big) d\te\\
& = \int_0^1  \Big( D\si (\te \tilde{\mathbb G}(x+y,t+\tau) + (1 -\te) \tilde{\mathbb H}(x +y,t+\tau) ) -D\si (\te \tilde{\mathbb G}(x,t) + (1 -\te) \tilde{\mathbb H}(x,t) )   \Big)\\
& \quad : \Big( {\mathbb G}(x+y, t+\tau) -  {\mathbb H}(x+y, t+\tau) \Big) d\te\\
& +  \int_0^1 D\si (\te \tilde{\mathbb G}(x,t) + (1 -\te) \tilde{\mathbb H}(x,t) ): \Big({\mathbb G}(x+y, t+\tau) -  {\mathbb H}(x+y, t+\tau) - {\mathbb G}(x, t) +  {\mathbb H}(x, t) \Big) d\te\\
& = I_1 + I_2.
\end{split}
\end{align*}
Here, 
\begin{align*}
|I_2| \leq  \|D\si\|_{L^\infty (0, \max(\|\tilde{\mathbb G}\|_{L^\infty(\RN)},  \|\tilde{\mathbb G}\|_{L^\infty(\RN)} )}     \Big| {\mathbb G}(x+y, t+\tau) -  {\mathbb H}(x+y, t+\tau) - {\mathbb G}(x, t) +  {\mathbb H}(x, t)  \Big|  
\end{align*}

For $I_1$, 
\begin{align*}
& D  \si \big(\te \tilde{\mathbb G}(x+y,t+\tau) + (1 -\te) \tilde{\mathbb H}(x +y,t+\tau) \big) -D\si \big(\te \tilde{\mathbb G}(x,t) + (1 -\te) \tilde{\mathbb H}(x,t) \big)   \\
& = \int_0^1 \frac{d}{d \ga}  D \si \Big(  \ga \big( \te \tilde{\mathbb G}(x+y,t+\tau) + (1 -\te) \tilde{\mathbb H}(x +y,t+\tau)    \big) + (1 -\ga) \big(  \te \tilde{\mathbb G}(x,t) + (1 -\te) \tilde{\mathbb H}(x,t) \big)     \Big) d\ga\\
& = \int_0^1    D^2 \si \Big(  \ga \big( \te \tilde{\mathbb G}(x+y,t+\tau) + (1 -\te) \tilde{\mathbb H}(x +y,t+\tau)    \big) + (1 -\ga) \big(  \te \tilde{\mathbb G}(x,t) + (1 -\te) \tilde{\mathbb H}(x,t) \big)     \Big)\\
& :: \Big(\te\big(\tilde {\mathbb G}(x+y, t +\tau)  - \tilde {\mathbb G}(x, t)  \big) + (1 -\te) \big( \tilde {\mathbb H}(x+y, t+\tau)  - \tilde {\mathbb H}(x,t)    \big)   \Big) d\ga.
\end{align*}
Hence, we have 
\begin{align*}
|I_1| & \leq \|D^2\si\|_{L^\infty(0, \max(\|\tilde{\mathbb G}\|_{L^\infty(\RN)},  \|\tilde{\mathbb G}\|_{L^\infty(\RN)} )}     
\Big(\tilde{\mathbb G}\|_{L^\infty(\RN)} +  \|\tilde{\mathbb H}\|_{L^\infty(\RN)} \Big)   \\
& \qquad  \times 
\Big( \big| \tilde {\mathbb G}(x+y, t +\tau)  - \tilde {\mathbb G}(x, t) \big| + \big|    \tilde {\mathbb H}(x+y, t+\tau)  - \tilde {\mathbb H}(x,t)  \big|  \Big) \big| {\mathbb G}(x+y, t+\tau) -  {\mathbb H}(x+y, t+\tau) \big|   \Big).
\end{align*}
Summing all  estimate, we have 
\begin{align*}
\| \si ({\mathbb G}) -\si({\mathbb H})\|_{\dot B^{s,\frac{s}2}_{p,q} } 
& \leq c\| \si (\tilde {\mathbb G}) -\si ( \tilde {\mathbb H})\|_{\dot B^{s,\frac{s}2}_{p,q} (\RN) }\\
  & \leq c    \Big(   \int_{{\mathbb R}^{n +1}}  
\frac1{(|y| +|\tau|^\frac12)^{n+2  + ps  } } \\
& \quad  \times  \big( \int_{{\mathbb R}^{n +1}}  |\si ( \tilde{\mathbb G}(x+y,t+\tau)) - \si ( \tilde{\mathbb H}(x +y,t+\tau) )\\
& \qquad \qquad  -  \si (\tilde{\mathbb G}(x,t))  + \si( \tilde{\mathbb H} (x,t) )  |^p  dxdt \big)^{\frac{q}p}  dyd\tau  \Big)^\frac1q\\
& \leq c \Big( \| \tilde{\mathbb G} - \tilde{\mathbb H}\|_{\dot B^{s, \frac{s}2}_{p,q} (\RN)} + \big( \tilde{\mathbb G}\|_{\dot B^{s, \frac{s}2}_{p,q} (\RN)} +  \|\tilde{\mathbb H}\|_{\dot B^{s, \frac{s}2}_{p,q} (\RN)}    \big) \big(  \| \tilde{\mathbb G}-\tilde{\mathbb H}\|_{L^\infty (\RN)}  \big) \Big)\\
& \leq c \Big( \|  {\mathbb G} -  {\mathbb H}\|_{\dot B^{s, \frac{s}2}_{p,q}  } + \big( \|  {\mathbb G}\|_{\dot B^{s, \frac{s}2}_{p,q}  } +  \| {\mathbb H}\|_{\dot B^{s, \frac{s}2}_{p,q} }    \big) \big(  \|  {\mathbb G}- {\mathbb H}\|_{L^\infty  }  \big) \Big).
\end{align*}
We complete the proof of (2) of  Lemma \ref{lemma0716}.

\section{Proof of Lemma \ref{lemma0601} }
\label{prooflemma0601}
\setcounter{equation}{0}

By  H$\ddot{\rm o}$lder's inequality and Besov  space embedding,  
\begin{align*}
\| u \|_{L^p  } & \leq  \| u \|^{1 -\te}_{L^\infty } \| u \|^{\te}_{L^{n+2} }
 \leq  c   \| u \|^{ 1 -\te}_{\dot B^{\frac{n+2}p, \frac{n+2}{2p}}_{p,1}  }  \| u \|^{\te}_{L^{n+2}  }.
\end{align*}
We complete the proof of \eqref{230519-4}.

Note that  $\dot B^{\frac{n+2}p, \frac{n+2}{2p}}_{p,1}   = ( L^p, \dot W^{1 + \frac{n+2}p, \frac12 + \frac{n+2}{2p}}_{p} )_{\eta,1}$ from (1) of Proposition \ref{prop0215}. Since $ \dot B^{1 + \frac{n+2}p, \frac12 + \frac{n+2}{2p}}_{p,1}  \subset \dot W^{1 + \frac{n+2}p, \frac12 + \frac{n+2}{2p}}_{p} $,  from above estimate,  
\begin{align*}
\| u \|_{\dot B^{\frac{n+2}p, \frac{n+2}{2p}}_{p,1} } &   \leq c \| u \|^{1-\eta}_{L^p  } \| u \|^{\eta}_{\dot W^{1 + \frac{n+2}p, \frac12 + \frac{n+2}{2p}}_{p}  } 
\leq c \big(  \| u \|^{ 1 -\te}_{\dot B^{\frac{n+2}p, \frac{n+2}{2p}}_{p,1}  }  \| u \|^{\te}_{L^{n+2}  }     \big)^{1-\eta} 
          \| u \|^{\eta}_{\dot B^{1 + \frac{n+2}p, \frac12 + \frac{n+2}{2p}}_{p,1}  }.
\end{align*}
Hence, 
\begin{align*}
\| u \|^{ \te + \eta -\te \eta  }_{\dot B^{\frac{n+2}p, \frac{n+2}{2p}}_{p,1}  }
& \leq c   \| u \|^{ \te (1- \eta)}_{L^{n+2} }  \| u \|^{\eta}_{\dot B^{1 + \frac{n+2}p, \frac12 + \frac{n+2}{2p}}_{p,1}  }.
\end{align*}
This implies  \eqref{230519-5}.

\section{Proof Lemma \ref{lemma1208}}
\label{prooflemma1208}
\setcounter{equation}{0}
Because the proofs are exactly same, we only prove the case of $\dot B^{s, \frac{s}2}_{p,q} $. 

If $ f \in  \dot B^{s, \frac{s}2}_{p,q} $, then $E_1 f \in \dot B^{s, \frac{s}2}_{p,q} (\R \times  (0, \infty))$ with $\| E_1 f\|_{\dot B^{s, \frac{s}2}_{p,q} (\R \times (0, \infty) )} \leq c \| f\|_{ \dot B^{s, \frac{s}2}_{p,q}}$, where the  extension operator $E_1$ is defined in Section \ref{notation}.   Hence, we prove Lemma \ref{lemma1208} for the case of $ f \in \dot B^{s, \frac{s}2}_{p,q} (\R \times  (0, \infty))$.

From (5) in Proposition \ref{prop2}, we have
\begin{align*}
\dot B^{s, \frac{s}2}_{p,q}  \subset L^\infty (0, \infty; \dot  B^{s -\frac2p}_{p,q} (\R)).
\end{align*}
Let $ \frac2p< s<2  $
 and $f \in \dot B^{s, \frac{s}2}_{p,q} $ with $ f(x,0) =0$ for $x \in \R$. Let $\tilde f(x,t) = f(x,t)$ for $ t > 0 $ and $\tilde f(x,t) =0$ for $ t < 0$. Accordingly, $ \tilde f \in \dot B^{s, \frac{s}2}_{p,q} ({\mathbb R}^{n+1}) $ with $ \| \tilde f\|_{\dot B^{s, \frac{s}2}_{p,q} (\R \times (-\infty, t) )} \leq c \| f\|_{\dot B^{s, \frac{s}2}_{p,q} (\R \times (0, t))   }$.  From (5) of Proposition \ref{prop2},  
\begin{align*}
\| f(t) \|_{\dot B^{s-\frac2p}_{p,q} (\R)} &\leq c \| \tilde f\|_{ \dot B^{s, \frac{s}2}_{p,q} (\R \times (-\infty, t) )  } \leq c \|  f\|_{ \dot B^{s, \frac{s}2}_{p,q} (\R \times (0, t))  } \ri 0 \quad \mbox{as} \quad t \ri 0.
\end{align*}

Let  $f_0 (x) = f(x,0)$.   From (5) of Proposition \ref{prop2},   $f_0 \in \dot B^{s -\frac2p}_{p,q} (\R)$. From well-known result of the Cauchy problem of the heat equation in $\R \times (0, \infty)$, the function defined by  $F(x,t) = \Ga_t * f_0(x)$ is in  $\dot B^{s, \frac{s}2}_{p,q}$ with $\| F \|_{\dot B^{s, \frac{s}2}_{p,q} } \leq c \| f_0 \|_{\dot B^{s -\frac2p}_{p,q} (\R)}$ and $\| F(t) - f_0 \|_{\dot B^{s -\frac2p}_{p,q} (\R)} \ri 0 $ as $t \ri 0$. Since $F -f \in \dot B^{s, \frac{s}2}_{p,q} $ and $(F - f)|_{t =0} =0$, from the above argument,  $\| F(t) - f (t) \|_{\dot B^{s -\frac2p}_{p,q} (\R)} \ri 0 $ as $t \ri 0 $. Accordingly
\begin{align*}
\| f(t) - f_0\|_{\dot B^{s-\frac2p}_{p,q} (\R)}  \leq
\| f(t) - F(t)\|_{\dot B^{s-\frac2p}_{p,q} (\R)}  + \| F(t) - f_0\|_{\dot B^{s -\frac2p}_{p,q} (\R)}  \ri 0 \quad \mbox{as} \quad t\ri 0.
\end{align*}
This implies
\begin{align*}
\dot B^{\al, \frac{s}2}_{p,q}  \subset C ([0, \infty); \dot  B^{s -\frac2p}_{p,q} (\R)).
\end{align*}
Hence, we proved Lemma \ref{lemma1208} for $ \frac2p < s< 2 $.

Let $  2k +\frac2p< s< 2k +2  $ for $k \in {\mathbb N}$. Accordingly $D_x^{2k} f \in \dot B^{s-2k, \frac{s}2 -k}_{p,q}$ and $D_x^{2k} f_0 \in \dot B^{s-2k -\frac2p}_{p,q} (\R)$. From the above argument,  
\begin{align*}
\| f(t) - f_0 \|_{\dot B_{p,q}^{s-\frac2p} (\R)} \leq \| D_x^{2k} ( f(t) -   f_0 ) \|_{\dot B_{p,q}^{s-2k -\frac2p} (\R)} \ri 0 \quad \mbox{as} \quad  t \ri 0.
\end{align*}
Hence, we proved Lemma \ref{lemma1208} for  $ \frac2p < s  $.

\section*{Acknowledgements}
 Chang (NRF-2020R1A2C1A01102531) and Jin (RS-2023-00280597)
are supported by the Basic Research Program
through the National Research Foundation of Korea
funded by Ministry of Science and ICT.

{\bf AUTHOR DECLARATIONS}\\
{\bf Conflict of Interest}\\
The authors have no conflicts to disclose.\\
{\bf Author Contributions } \\
Tongkeun Chang: Writing – original draft (equal).  Bum Ja Jin:  Writing – original draft (equal).\\ 
{\bf  DATA AVAILABILITY}\\
The data that support the findings of this study are available from the corresponding author upon reasonable request.

\end{document}